\documentclass[11pt]{amsart}
\usepackage{geometry} 
\usepackage{pgfplots}
\usepackage{psfrag}
\usepackage{amsmath}
\usepackage{calc}
\usepackage{mathrsfs}
\usepackage{systeme}
\usepackage{amsfonts}
\usepackage{amsthm}
\usepackage{algorithm2e}
\usepackage[applemac]{inputenc}
\usepackage{bbold}
\usepackage[numbers]{natbib}

\newcommand{\be}{\begin{equation}}
\newcommand{\ee}{\end{equation}}
\newcommand{\ba}{\begin{aligned}}
\newcommand{\ea}{\end{aligned}}

\newcommand{\R}{\mathbb{R}}

\newcommand{\scr}{\mathscr}

\newcommand{\mpo}{\mathcal{P}(\mathcal O)}
\newcommand{\mppo}{\mathcal{P}_p(\mathcal O)}
\newcommand{\mpprd}{\mathcal{P}_p(\R^d)}

\newcommand{\mpt}{\mathcal{P}_2(\R^d)}

\newcommand{\mo}{\mathcal{O}}
\newcommand{\eps}{\epsilon}

\newcommand{\ol}{\overline}
\newcommand{\mcl}{\mathcal}

\newtheorem{Theorem}{Theorem}[section]

\newtheorem{Lemma}[Theorem]{Lemma}
\newtheorem{Rem}[Theorem]{Remark}
\newtheorem{Def}[Theorem]{Definition}
\newtheorem{Prop}[Theorem]{Proposition}

\usepackage{appendix}
\usepackage{chngcntr}
\usepackage{etoolbox}
\usepackage{xcolor} 
\usepackage{hyperref}

\hypersetup{
    colorlinks=true, 
    linkcolor=blue,  
    citecolor=blue,  
    urlcolor=blue    
}

\newcommand{\triplenorm}[1]{{\left\vert\kern-0.25ex\left\vert\kern-0.25ex\left\vert #1 
    \right\vert\kern-0.25ex\right\vert\kern-0.25ex\right\vert}}

\title{Wasserstein spaces have the Radon-Nikodym property}

\begin{document}





\author{Charles Bertucci \textsuperscript{a}}
\thanks{ \textsuperscript{a} CEREMADE, CNRS, UMR 7534, Universit\'e Paris Dauphine-PSL, 75016 Paris, France.}

\maketitle

\begin{abstract}
Wasserstein spaces carry two natural geometries: a vertical one, inherited from the linear structure of signed measures, and a horizontal one, induced by couplings. We prove that a Radon-Nikodym property holds in both. In the vertical geometry, bounded subsets of $\mpprd$ are dentable in an ambient Banach space which itself fails the Radon-Nikodym property; we derive variational principles, Radon-Nikodym theorems for vector measures and martingale convergence results. In the horizontal geometry, we show that a Radon-Nikodym property is inherited from the one of spaces of random variables. We introduce horizontal martingales and prove that they are exactly the laws of $L^p$-valued martingales.
\end{abstract}



\setcounter{tocdepth}{1}

\tableofcontents
\section*{Introduction}
\subsection*{General introduction}The Radon-Nikodym property (RNP) is a geometric property usually stated for Banach spaces. It has many equivalent definitions and formally, it can be said that the RNP holds in a Banach space $X$ when the Radon-Nikodym theorem holds for $X$-valued vector measures. More geometrically, a space $X$ satisfies the RNP if any bounded subset has non-empty slices (a cut by a continuous linear form) of arbitrarily small diameters. Precise definitions shall be given later on. The links between the RNP and geometric properties of Banach spaces has been long identified, notably through the works of Phelps (see Section 5 in \citep{phelps}) and Bourgain \citep{bourgain1,bourgain2}, and links with convergence of martingales were made explicit \citep{diestel}. For a complete presentation of the RNP and its links with different areas of mathematics, we refer to Diestel and Uhl \citep{diestel}.\\

In our opinion, there are two main geometries on Wasserstein spaces: a flat (or vertical) one originating from the linear structure of the vector space of signed measures, and a horizontal one arising from the geometry of the space on which the Wasserstein space is defined. We refer to Bertucci \citep{bertuccibook} for a presentation of this point of view. Our main contribution is to prove that the RNP holds generally for the vertical geometry in Wasserstein spaces, and that, if $p\in (1,\infty)$, it also holds in the horizontal geometry for the $p$ Wasserstein space over $\R^d$. Note that in neither of the two cases, Wasserstein spaces will be seen as Banach spaces.\\

In the vertical geometry, we shall embed Wasserstein spaces into certain Banach spaces, very similar to Lipschitz free spaces. We will prove that, even if the larger Banach spaces do not satisfy the RNP, it holds in the Wasserstein spaces. We will prove that the property is satisfied by showing the dentability of bounded sets, i.e. by constructing slices of small diameters directly. Usual consequences of the RNP are then shown to hold in this geometry.\\

 In the horizontal geometry, we adopt a more intrinsic point of view. Following the systematic approach developed by P.-L. Lions in \citep{lions20112012}, we lift the problem into appropriate spaces of random variables, which are well known to satisfy the RNP. Pulling down this information at the level of the Wasserstein space, we shall obtain non-linear versions of the RNP. Notably, by using this lifting/pulling-down procedure, we are able to define and study a notion of (horizontal) martingale on Wasserstein spaces, even though there is no linear structure on such spaces to define the conditional expectation.
 
 \subsection*{Motivations}
 The original motivation for the results presented in this paper stems from several problems in analysis and probability theory. While the present work is primarily theoretic, it is driven by three main areas of applications. 
 
 The first one is analysis on Wasserstein spaces, especially partial differential equations, where variational principles, in the spirit of Ekeland's variational principle \citep{ekeland}, are necessary to establish rigorously various results. Because of the nature of the problem of interest, more precise results than Ekeland's are often necessary. For the moment, quite involved results were presented in \citep{cosso2021master,bayraktar}, and a simpler result providing less regularity for the perturbation was proved in \citep{nonstegall}. Theorem \ref{thm:mainstegall} below, is a more powerful result, whose proof can be understood simply through the RNP. It is established by means of the vertical geometry mentioned above, but it is almost a purely topological statement on Wasserstein spaces (almost because a notion of convex set is used). 
 
 A second problem of interest is the definition of martingales in Wasserstein spaces, which are not vector spaces and thus ill set for the use of conditional expectation. As stochastic processes in Wasserstein spaces are becoming more and more omnipresent in the mathematical literature \citep{cdll,lacker,sturm2}, defining martingales valued in these spaces is a natural question. In the vertical geometry, we introduced a framework in which Wasserstein spaces can be seen as subsets of Banach spaces, thus enabling the use of standard notions and tools. In the horizontal geometry, we introduce a new definition in the spirit of the ones used to define martingales on manifolds \citep{schwartz}. In both cases, we established mean convergence results (Theorems \ref{thm:charles} and \ref{thm:martconv}).
 
 The third direction this paper was motivated by is the construction of a reference measure on Wasserstein spaces. Such a direction of research is quite active at the moment \citep{sturm3,sturm2,schiavo} and the current work does not exactly contribute to it, but we believe they are not entirely disconnected. Such a reference measure could be naturally obtained as the long time limit of certain measure-valued processes and we believe Theorem \ref{thm:martconv} below can be used in such directions. 
 
\subsection*{Organization of the manuscript}
The rest of the paper is organized as follows. In the preliminaries, we recall some usual results and definitions about Wasserstein spaces and we introduce the two geometries of interest. We also recall some basic facts about the RNP. In Part \ref{part:vert}, we establish a vertical version of the RNP, namely by constructing, in Section \ref{sec:slices}, slices of small diameters of bounded subsets. We then draw the usual conclusions from this fact: some perturbed optimization results, analytic versions of the Radon-Nikodym Theorem for vector measures and convergence of martingales in respectively Sections \ref{sec:optim}, \ref{sec:analytic} and \ref{sec:mart1}. Those conclusions follows easily once the framework of Section \ref{sec:fa} is put in place. Extensions are discussed in Section \ref{sec:extension}. In Part \ref{part:hor}, by lifting the problems at hand to suitable sets of random variables, we establish some geometric properties of convex sets in the associated geometry (Section \ref{sec:geohor}) and, in Section \ref{sec:marthor}, define what we believe is a new notion of martingale for which we prove a mean convergence result.

\section*{Preliminaries}
For $p \in (1,\infty)$, we use the notation $p'=\frac{p}{p-1}$. For a tuple $(x_1,\dots,x_n)$, $\pi_i$ stands for the projection of the $i$-th component. A point of minimum of a function defined on a subset of a metric space is said to be strongly exposed if all the minimizing sequences converge toward this minimum (for the underlying distance). A point of a subset of a Banach space is strongly exposed if it is a strongly exposed point of minimum of a linear functional on this set. The Sobolev spaces of order $k$ with integrability coefficient $p=2$ on a space $\mo$ are denoted by $H^k(\mo)$.

\subsection*{Wasserstein spaces}Given a Polish (i.e. complete separable metric) space $(\mo,d)$, and $p\in [1,\infty)$, the $p$-Wasserstein space over $\mo$ is the set $\mppo$ defined by 
$$
\mppo : = \left\{ µ \in \mathcal P(\mo) \bigg | \,\exists x_0 \in \mo, \int_\mo (d(x,x_0))^pµ(dx) < \infty\right\},
$$
equipped with the distance $W_p$ given by
$$
W_p(µ,\nu) = \left(\inf_{\gamma \in \Pi(µ,\nu)} \int_{\mo \times \mo}d(x,y)^p\gamma(dx,dy) \right)^\frac1p,
$$
where $\Pi(µ,\nu)$ stands for the set of couplings between $µ, \nu \in \mpo$. For a proof of the fact that $(\mppo,W_p)$ is itself a Polish space, we refer to \citep{villani}.

For $p \in [1,\infty)$ and $µ \in \mpprd$, we write $M_p(µ)= \int_{\R^d}|x|^pµ(dx)$. For $p \in (1,\infty)$ we also define the map $\mcl I_p: \mpprd \times \mcl P_{p'}(\R^d)\mapsto \R$ by
\be\label{defI}
\mcl I_p(µ,\nu) = \inf_{\gamma \in \Pi(µ,\nu)} - \int_{\R^d\times \R^d} x\cdot y\, \gamma(dx,dy).
\ee
Given two measurable sets $X$ and $Y$, a measure $µ$ on $X$ and a measurable map $f: X\mapsto Y$, the image measure of $µ$ by $f$ is denoted by $f_\#µ$.

For $p \in [1,\infty)$, we define the set of $\R^d$ valued random variables on the probability space $([0,1],\mcl B,Leb)$ with finite $p$ moment by $\mathbb L^p$. Note that it is a Banach space equipped with the norm $\|X\|_p = (\mathbb E[|X|^p])^\frac1p$, which is reflexive when $p \in (1,\infty)$.

Given $µ,\nu \in \mcl P_1(\R^d)$, we say that $µ$ is smaller than $\nu$ in the convex order and we note $µ\leq_c \nu$ if for any convex function $\phi:\R^d \mapsto \R$, $\int_{\R^d} \phi dµ \leq \int_{\R^d}\phi d\nu$.

\subsection*{RNP}
Let $X$ be a Banach space. We follow the definition of Phelps \citep{phelps} and say that a subset $A \subset X$ has the RNP if, every bounded subset $B \subset A$ is dentable. The dentability of $B$ means that for every $\eps > 0$, there exists $y \in X'$ and $\alpha \in \R$ such that $\{x \in B | \langle y,x \rangle \leq \alpha\}$ is non empty and has diameter at most $\eps$. Such a property implies strong results, such as Stegall's Lemma (Theorem 5.15 in \citep{phelps}) on any closed bounded convex subset with the RNP, and when the space $X$ itself has the RNP: usual forms of the Radon-Nikodym Theorem (Part III in \citep{diestel}) or mean convergence of martingales (Corollary V.2.4 in \citep{diestel}). In fact, it is a routine check that we will do below to prove that local versions of those results can be obtained, provided a careful restriction, if the RNP is only valid on certain subsets of $X$.\\

A typical Banach space failing the RNP is the Lebesgue space $L^1([0,1],\R)$, whose unit ball is not dentable.

\subsection*{Two convexities}
Let $A\subset \mpprd$. Classically, the set $A$ is said to be convex if for all $µ,\nu \in A$, $t\in [0,1]$, $(1-t) µ +t \nu \in A$. It is said to be coupling convex if for all $µ,\nu \in A$, $\gamma \in \Pi(µ,\nu)$, $t\in [0,1]$, $((1-t)\pi_1 + t\pi_2)_\#\gamma \in A$. We denote by $co(A)$ the convex hull of $A$ and by $cco(A)$ the coupling convex hull of $A$.

Analogously, a function $U: \mpprd \mapsto \R$ is said to be convex if for all $µ,\nu \in A$, $t\in [0,1]$, $U((1-t) µ +t \nu ) \leq (1-t)U(µ) + t U(\nu)$. It is said to be coupling convex if for all $µ,\nu \in A$, $\gamma \in \Pi(µ,\nu)$, $t\in [0,1]$, $U(((1-t)\pi_1 + t\pi_2)_\#\gamma) \leq (1-t) U(µ) + t U(\nu)$.

The coupling convexity is the geometry of the lift. Indeed let $U: \mpprd \mapsto \R$ be coupling convex, and define its lift $\mcl U : \mathbb L^p \mapsto \R$ by $\mcl U(X) = U(\mcl L(X))$. Then $\mcl U$ is convex. For $p \in (1,\infty)$, the function $\mcl I_p$ defined in \eqref{defI} is coupling concave in each of its arguments.
Coupling convexity behaves quite well with the convex order of probability measures. For instance, we can find in \citep{bellini2021law} the following results, which are easy consequences of a result of Ryff \citep{ryff} on the set of random variables with the same law.
\begin{Prop}\label{prop:preliminaries}
Let $p \in [1,\infty)$.
\begin{itemize}
\item For all $µ \in \mpprd$, $\ol{cco}(\{µ\}) = \{\nu \in \mpprd, \nu \leq_c µ\}.$
\item Let $\phi: \mpprd \mapsto \R$ be coupling convex and $µ,\nu \in \mpprd, µ\leq_c \nu$. Then $\phi(µ) \leq \phi(\nu)$.
\end{itemize}
\end{Prop}
 For more details on coupling convexity, we refer the reader to \citep{bertuccibook}, or to \citep{pinzi,cavagnari} where it is called total convexity.

In all what follows, on Wasserstein spaces, the flat or usual convexity is what we call the vertical geometry whereas the coupling convexity is the convexity of the horizontal geometry.

\part{Vertical geometry}\label{part:vert}
To lighten the notation of this section, we work with the set $\mpt$, equipped with the $W_2$ distance, and discuss extensions at the end of this part.

\section{A geometric property}\label{sec:slices}
The main result of this section is the following.
\begin{Theorem}\label{thm:1}
For any non-empty, bounded, closed, convex subset $\mathscr C \subset \mpt$, there exists $µ \in \mathscr C$ such that, for any $\eps > 0$, $µ\notin \ol{co}(\mathscr C \setminus B(µ,\eps))$.
\end{Theorem}

Before giving the proof of Theorem \ref{thm:1}, we present the next easy result to give an idea of what we are trying to prove.

\begin{Prop}\label{prop:dirac}
For any $\mathscr C \subset \mpt$, $x \in \R^d$ such that $\delta_x \in \mathscr C$, $\eps > 0$, it holds that $\delta_x \notin \ol{co}(\mathscr C \setminus B(\delta_x,\eps))$.
\end{Prop}
\begin{proof}
For all $\eps > 0$, $\nu \in \mathscr C\setminus B(\delta_x,\eps)$,
$$
\int_{\R^d}|y-x|^2\nu(dy) = W_2^2(\nu,\delta_x) \geq \eps^2.
$$
Hence, by linearity and continuity (for $W_2$) of the integral above in $\nu$, for any $\nu \in \ol{co}(\mathscr C \setminus B(\delta_x,\eps))$, the previous inequality also holds. So $\delta_x \notin \ol{co}(\mathscr C \setminus B(\delta_x,\eps))$.
\end{proof}
This proposition makes clear that Dirac masses are extreme points in this geometry. In fact, they are dents or strongly exposed points. In the general case, i.e. when $\mathscr C$ does not contain any Dirac mass, we shall use a similar argument, namely by trying to select a measure which does not spread mass too much, to prove that a dent exists.\\

We start with a lemma.
\begin{Lemma}\label{lemma:comp}
Let $\mathscr C \subset \mpt$ be a compact set. Then, for any integer $k > \frac d2$, there exists a dense $G_\delta$ set $\mathcal A \subset H^k(\R^d)$ such that for any $\phi \in \mcl A$, $µ \mapsto \langle \phi,µ\rangle$ has a unique point of minimum on $\mathscr C$.
\end{Lemma}
\begin{Rem}
\begin{itemize}
\item Because $\mathscr C$ is compact, a unique point of minimum or a strongly exposed point of minimum are equivalent (for a lsc function).
\item Recall that since $k > \frac d2$, $H^k(\R^d) \subset \mcl C^\alpha_0(\R^d)$, for some $\alpha \in (0,1)$, the space of H\"older continuous functions which vanish at infinity.
\item For any $\phi \in \mcl A$, the unique minimizer $µ_\phi$ of $µ \mapsto \langle \phi,µ\rangle$ is in fact such that for every $\eps > 0$, $µ_\phi \notin \ol{co}(\mathscr C \setminus B(µ_\phi,\eps))$. This follows, as was the case in Proposition \ref{prop:dirac}, from the linearity (concavity is in fact sufficient) and continuity (upper semi continuity is sufficient) of the map $µ \mapsto \langle \phi,µ\rangle$.
\end{itemize}
\end{Rem}
\begin{proof}
The argument is an immediate adaptation of Lemma 2.1 in \citep{bertucci2023monotone}, that we present for the sake of completeness. Let $\Psi$ be defined by
$$
\ba
\Psi: & H^k(\R^d) \mapsto 2^\mathscr C,\\
& \phi \mapsto \left\{µ \in \mathscr C| \nu \mapsto \langle \phi,\nu\rangle \text{ is minimum on $\mathscr C$ at } µ\right\}.
\ea
$$
Observe that $\mathscr C \subset (H^k(\R^d))'\cong H^k(\R^d)$. The operator $-\Psi$ is cyclical monotone. Indeed, consider a sequence $(µ_i,\phi_i)_{1 \leq i \leq n}$ such that $µ_i \in \Psi(\phi_i)$ for all $1 \leq i \leq n$. Then
$$
\sum_{i =1}^n\langle \phi_{i},µ_{i} - µ_{i-1}\rangle \leq 0,
$$
by optimality of the $µ_i$, where we set $µ_{0} = µ_n$. Hence, we deduce from Rockafellar's Theorem that $-\Psi \subset \partial F$ where $F: H^k(\R^d)\mapsto \R$ is a proper convex function and $\partial F$ its sub-differential. Because $\Psi$ has full domain so has $F$. Furthermore, because of the construction of $F$, we know it is lsc, hence continuous everywhere. Since $H^k(\R^d)$ is a separable Hilbert space, it is an Asplund space. We thus deduce that $F$ is Fr\'echet differentiable on a dense $G_\delta$ set. On this set, $\Psi$ can only be single valued, hence the result.
\end{proof}
We are now ready to prove the main result of this section.

\begin{proof}[Proof of Theorem \ref{thm:1}]
Let $µ_0 \in \mathscr C$. Using de la Vall\'ee-Poussin criterion, we know that there exists a map $\psi: \R^d \mapsto \R$ such that
\be\label{eq:supq}
\lim_{|x|\to \infty}\frac{\psi(x)}{1 + |x|^2} = + \infty,
\ee
$$
\int_{\R^d} \psi(x)µ_0(dx) < \infty.
$$
Without loss of generality, we can always assume that $\psi$ is smooth, radial and such that $\frac{\psi(x)}{|x|^2}$ is non-decreasing in $|x|$. Consider now $K_\psi$ defined by 
$$
K_\psi := \text{argmin}\left\{\int_{\R^d}\psi(x)µ(dx) \bigg | µ \in \mathscr C \right\}.
$$
By Prokhorov's Theorem and the lsc of $µ \mapsto \langle \psi,µ\rangle$, $K_\psi$ is a compact of $\mathscr C$. From Lemma \ref{lemma:comp}, there exists $\xi \in \mathcal C^2_0(\R^d)$ (i.e. $\mcl C^2$ functions converging toward $0$ at infinity) such that $µ \mapsto \langle \xi,µ\rangle$ has a unique (thus strongly exposed) point of minimum on $K_\psi$. Denote this point by $\ol µ$.

Let $\eps >0$, we want to show that there exist $\kappa,\delta > 0$ such that
\be\label{eq:1}
\inf_{µ \in \mathscr C \setminus B(\ol µ,\eps)} \int_{\R^d} (\psi + \delta \xi) dµ \geq \kappa + \int_{\R^d}( \psi + \delta \xi )d\ol µ.
\ee
Assume it is not the case. Then, there exists a sequence $(µ_n)_{n \geq 1}$ valued in $\mathscr C \setminus B(\ol µ,\eps)$ such that for all $n \geq 1$
\be\label{eq:2}
\int_{\R^d} \left(\psi + \frac1n \xi\right) dµ_n \leq \frac{1}{n^2} + \int_{\R^d}\left( \psi + \frac1n \xi \right)d\ol µ.
\ee
In particular, thanks to Prokhorov's Theorem and \eqref{eq:supq}, we know that $(µ_n)_{n \geq 1}$ has a limit point $µ^*$ (for $W_2$), which is necessarily in $\mathscr C \setminus B(\ol µ,\eps)$. Furthermore, $µ^* \in K_\psi$ by passing to the limit in \eqref{eq:2}. Since $W_2(µ^*,\ol µ) \geq \eps$, we deduce from the strong exposure of $\ol µ$ that there exists $\alpha > 0$ such that 
\be\label{eq:3}
\int_{\R^d}\xi dµ^* \geq \alpha + \int_{\R^d} \xi d\ol µ.
\ee
Since $\ol µ \in K_\psi$, we can use $\langle \psi,\ol µ \rangle \leq \langle \psi,µ_n \rangle$ in \eqref{eq:2} to obtain (after a multiplication by $n$)
$$
\int_{\R^d} \xi dµ_n \leq \frac{1}{n} + \int_{\R^d}\xi d\ol µ,
$$
which contradicts \eqref{eq:3} in the limit such that $µ_n \to µ^*$. Hence \eqref{eq:1} holds for some $\delta,\kappa >0$. Observe that, because in general $µ\mapsto \langle \psi ,µ\rangle$ is not upper-semi continuous for $W_2$, we are not able to deduce the result directly from \eqref{eq:1}. To obtain the result, we shall prove that for $\eps,\kappa,\delta > 0$ chosen as above, there exist $M> 0$, such that
\be\label{eq:41}
\inf_{µ \in \mathscr C \setminus B(\ol µ,\eps)} \int_{\R^d} (\psi_M + \delta \xi) dµ \geq \frac\kappa2 + \int_{\R^d}( \psi_M + \delta \xi )\ol dµ,
\ee
where $\psi_M: \R^d \mapsto \R$ is a smooth radial function satisfying 
$$
\psi_M(x) := \begin{cases} \psi(x) \text{ if } |x| \leq M,\\ \frac{\psi(M)}{M^2}|x|^2 \text{ if } |x| \geq M+1,\end{cases}
$$
as well as $\psi_M(x) \geq \frac{\psi(M)}{M^2}|x|^2$ when $|x| \geq M$ and $\R_+ \mapsto \R, \lambda \mapsto \psi_M(\lambda x)$ is smooth and increasing for any $x \in \R^d$. Once again, we argue by contradiction and consider a sequence $(µ_n)_{n \geq 1}$ in $\mathscr C \setminus B(\ol µ,\eps)$ such that for any $n\geq 1$
\be\label{eq:4}
\int_{\R^d} \left(\psi_n + \delta \xi\right) dµ_n \leq \frac\kappa2 + \int_{\R^d}\left( \psi_n + \delta \xi \right)d\ol µ.
\ee
Since $(µ_n)_{n \geq 1}$ is bounded in $\mpt$, it has a weak limit, which we denote by $µ^*$. Extracting a subsequence if necessary, we assume that the whole sequence converges toward $µ^*$. 
Observe that $(\psi_n)_{n \geq 1}$ converges monotonically toward $\psi$, so that $\lim_{n \to \infty} \langle \psi_n,\ol µ\rangle = \langle \psi,\ol µ\rangle$. Remark that
$$
\ba
\int_{\R^d} \psi_n dµ^* \leq \liminf_{m \to \infty} \int_{\R^d}\psi_n dµ_m \leq  \liminf_{m \to \infty} \int_{\R^d}\psi_m dµ_m \leq \int_{\R^d}\psi d \olµ + \frac\kappa2 + \delta \int_{\R^d}\xi(\ol µ-µ^*),
\ea
$$
where we used the lsc of $µ \mapsto \langle \psi_n,µ\rangle$ for the weak topology in the first inequality, that $(\psi_n)_{n \geq 1}$ is a non-decreasing sequence and \eqref{eq:4}.

Hence, thanks to Beppo Levi's Theorem, 
$$
\langle \psi,µ^*\rangle \leq \langle \psi,\ol µ\rangle + \frac\kappa2 + \delta \int_{\R^d}\xi(\olµ-µ^*).
$$
 To conclude, it suffices to remark that $(µ_n)_{n \geq 1}$ converges toward $µ^*$ in $\mpt$, so that $µ^* \in \scr C\setminus B(\ol µ,\eps)$ and obtain a contradiction in \eqref{eq:1}. This convergence in $W_2$ holds since, by construction, for $n \geq M$
$$
\int_{|x|> M} |x|^2µ_n(dx) \leq \frac{M^2}{\psi(M)}\int_{\R^d}\psi_M dµ_n \leq \frac{M^2}{\psi(M)}\int_{\R^d}\psi_n dµ_n \leq \frac{M^2}{\psi(M)}C \underset{M \to \infty}{\longrightarrow} 0.
$$
This uniform integrability of $(µ_n)_{n \geq 1}$ implies that $\int_{\R^d}|x|^2µ_n(dx) \to_{n \to \infty} \int_{\R^d}|x|^2µ^*(dx)$, from which the convergence in $\mpt$ follows. Thus, there exists $M >0$ so that \eqref{eq:41} holds, hence the result since $\ol µ$ satisfies the claim, because $µ \mapsto \langle \psi_M + \delta \xi,µ\rangle $ is usc and concave.
\end{proof}
The main advantage of the previous proof is that the function used to approximately separate $\ol µ$ defines a continuous linear function on $\mpt$. This remark is the core idea to define the linear setting of the next section.

\section{Functional analysis framework}\label{sec:fa}
In this section, we introduce a variant of the Lipschitz free space (which is the predual of the Banach space of Lipschitz functions which vanish at a fixed point, see \citep{lfs}). Our extension will contain $\mpt$ "nicely". In order to do so, we start by working with the Sobolev space $X := W^{2,\infty}(\R^d)$ of functions with a Hessian in $L^\infty$, equipped with the norm
$$
\|\phi\|_X := \sup_{x \in \R^d}\left\{\frac{|\phi(x)|}{1 + |x|^2} \right\} + \|D^2\phi\|_\infty.
$$
We do not detail that $X$ is indeed a Banach space, and we recall that if $\|\phi\|_X$ is finite, then $\nabla_x \phi$ grows at most linearly, with a constant depending only on $\|\phi\|_X$. We introduce the set $E$, which we define by completing the span of Dirac masses for the norm
$$
\|µ\|_E := \sup_{\|\phi\|_X \leq 1} \int_{\R^d}\phi dµ.
$$
The following result is standard.
\begin{Prop}
The space $(E,\|\cdot\|_E)$ is a Banach space and its topological dual is $X$. Furthermore, $\mpt$ is a closed convex subset of $E$ (for $\|\cdot\|_E$).
\end{Prop}
\begin{proof}
The space $E$ is clearly a Banach space. Let $L \in E'$ and define $\phi: \R^d \mapsto \R$ by $\phi(x) = L(\delta_x)$. Observe that for any $x,y \in \R^d$,
$$
\phi(x) - \phi(y) = L(\delta_x-\delta_y)\leq \|L\|_{E'} \sup_{\|f\|_X \leq 1} |f(x) - f(y)|.
$$
Hence, $\phi$ is locally Lipschitz continuous. Note that for some $c > 0$ depending only on $d$, $\|\delta_x\|_E \leq c (1+|x|^2)$. We now use this estimate to obtain
$$
\sup_{x \in \R^d} \frac{|\phi(x)|}{1+|x|^2} \leq \sup_{µ\in E\setminus\{0\}} \frac{|L(µ)|}{\|µ\|_E} < \infty.
$$
Evaluating $L$ on $\delta_{x+h}+ \delta_{x-h} -2\delta_x$, for $x,h\in \R^d$, easily implies that $\phi$ is $\mcl C^1$. We can then evaluate $L$ on $t^{-1}(\delta_{x+tv}-\delta_x -(\delta_{y+tv}-\delta_y))$ for $t > 0$, $x,y,v \in \R^d$ leads to
$$
\frac{\phi(x+tv) - \phi(x)}{t} - \frac{\phi(y+tv)-\phi(y)}{t} \leq \|L\|_{E'}|x-y|\,|v|.
$$
Taking the limit $t \to 0$ implies that $\phi \in X$. 

Obviously every $\phi \in X$ can be associated to an element of $L_\phi \in E'$ defined on Dirac masses by $L_\phi(\delta_x) = \phi(x)$, and the associated map $\phi \mapsto L_\phi$ is an embedding. Therefore we have indeed $E' = X$.

We now want to show that $\mpt$ is closed in $E$. Let $(µ_n)_{n \geq 0}$ be valued in $E\cap \mpt$, converging toward $\ol µ$ in $E$. In particular we deduce that for all $\phi \in X$,
$$
\int_{\R^d}\phi \,dµ_n \underset{n \to \infty}{\longrightarrow} \int_{\R^d}\phi \,d\ol µ.
$$
Since $x \mapsto |x|^2$ is in $X$, we deduce that $(µ_n)_{n \geq 0}$ is bounded in $\mpt$ (with convergence of the second order moment). By Prokhorov's Theorem, the previous implies the convergence in $\mathcal P_1(\R^d)$ with convergence of the second moment. Thus the convergence in $\mpt$ holds. 

We finally turn to the proof of $\mpt\subset E$. Let $µ \in \mpt$. It can be obtained as the limit (in $\mpt$) of empirical measures. Let us denote by $(µ_n)_{n \geq 0}$ such a sequence, which is clearly in $E$. Now observe that for all $n, m \geq 0$ and $\gamma$ an optimal coupling for $W_2(µ_n,µ_m)$,
$$
\ba
\|µ_n -µ_m\|_E &= \sup_{\|\phi\|_X \leq 1} \left\{\int_{\R^d}\phi(x) µ_n(dx) - \int_{\R^d}\phi(y) µ_m(dy)\right\}\\
& = \sup_{\|\phi\|_X \leq 1} \left\{\int_{\R^d\times \R^d}\phi(x)- \phi(y)\, \gamma(dx,dy)\right\}\\
&\leq  \int_{\R^d \times \R^d} |x-y| (1 + |x| + |y|) + |x-y|^2 \gamma(dx,dy)\\
& \leq \left(1 + \sqrt{M_2(µ_n)} + \sqrt{M_2(µ_m)}\right)W_2(µ_n,µ_m) + W_2^2(µ_n,µ_m).
\ea
$$
Notably, $(µ_n)_{n \geq 0}$ is a Cauchy sequence in $E$ as well and thus $µ \in E$.
\end{proof}
\begin{Rem}
This construction is classical as it is the one used to consider Lipschitz free spaces \citep{aliaga2026,lfs}, when $X$ is the set of Lipschitz functions and $\|\phi\|_X$ is replaced by 
$$
\sup_{x \in \R^d}\frac{|\nabla_x \phi(x)|}{1+|x|} + \|\nabla_x \phi\|_\infty.
$$
\end{Rem}
In the previous proof was made explicit the fact that the topologies induced by $\| \|_E$ or $W_2(\cdot,\cdot)$ are the same on $\mpt$. Note that there exists $C> 0$ depending only on $d$ such that for any $µ \in \mpt$,
$$
C^{-1}M_2(µ) \leq \|µ\|_E \leq 2 M_2(µ) +  \sqrt{M_2(µ)},
$$
where the second inequality follows from the computation of the previous proof. In particular, we have proven the following.
\begin{Prop}
Let $\mathscr C \subset \mpt$. Then $\mathscr C$ is closed (resp. convex, resp. bounded) in $\mpt$ if and only if it is closed (resp. convex, resp. bounded) in $E$.
\end{Prop}

The set $E$ is much larger than the set of measures, as the next result highlights.
\begin{Prop}\label{prop:exT}
For any $\rho \in L^\infty$ with compact support, the distribution 
$$
\mathcal T_\rho: \phi \mapsto \int_{\R^d}\Delta \phi(x) \rho(x)dx
$$
belongs to $E$.
\end{Prop}
\begin{proof}
For $h > 0$ and $\phi \in X$, set 
$$
\Delta_h \phi(x) := \frac{1}{h^2}\sum_{i=1}^d(\phi(x+he_i) + \phi(x - h e_i) - 2 \phi(x)).
$$
Because of the continuity in $h$ of the translation by $h e_i$ for functions in $L^1$, it is a simple analysis exercise to verify that for every $g \in L^1$,
$$
\lim_{h \to 0}\sup_{\|\phi\|_X\leq 1}\left |\int_{\R^d}\Delta_h \phi g - \int_{\R^d} \Delta \phi g \right| = 0.
$$
Because $\rho$ has a support bounded by some $R> 0$, there exists some $C >0$, such that any $\phi \in X$ such that $\|\phi\|_X \leq 1$ is $C$ Lipschitz on the support of $\rho$. Thus, using
$$
\int_{\R^d}\Delta_h \phi \rho = \int_{\R^d} \phi \Delta_h \rho,
$$
as well as $\Delta _h \rho \in L^\infty$ with compact support, for any $h > 0$, $\eps > 0$ there exists a linear combinations of Dirac masses, that we call $µ$ such that,
$$
\sup_{\|\phi\|_X \leq 1}\left | \int_{\R^d} \phi \Delta_h \rho - \int_{\R^d}\phi dµ\right | \leq \eps.
$$
(The Lipschitz constant of $\phi$ typically arise in such estimate, but it is here bounded uniformly). Hence, fixing first $h$ sufficiently small, then choosing sufficiently many Dirac masses, we have indeed proven that $T_\rho \in E$.
\end{proof}

We now recall that, given a subset $\mathscr C$ of $E$, a slice of $\mathscr C$ in $E$ is a set of the form
$$
S(\mathscr C,\phi,\alpha) := \{µ \in \mathscr C, \langle \phi, µ\rangle \leq \alpha\},
$$
for $\phi \in E'$ and $\alpha \in \R$. A set $\mathscr C\subset E$ is called dentable if it possesses non-empty slices of arbitrary small diameters. We leave as an exercise to verify that a bounded set $\mathscr C$ is dentable if and only its closed convex hull is dentable.

We now reformulate, in the space $E$, the geometric property obtained in the previous section.
\begin{Prop}\label{prop:dent}
The set $\mpt$ has the RNP (in $E$), as any bounded set of $\mpt$ is dentable in $E$.
\end{Prop}
\begin{proof}
Let $\mathscr C$ be a closed bounded convex set of $\mpt$. The proof of Theorem \ref{thm:1} yields that $\mathscr C$ is dentable, thus the result also holds for any bounded set. Indeed, using the notation of this proof, we precisely established that for any $\eps > 0$, we can find $M, \delta >0$ such that
$$
S\left(\mathscr C, \psi_M + \delta \xi, \lambda + \frac\kappa2\right)
$$
has a diameter smaller than $\eps$ and contains $\ol µ$, where $\lambda := \langle \psi_M + \delta \xi, \ol µ\rangle$. It suffices to observe that $\psi_M + \delta \xi \in X$.
\end{proof}
Observe that this property is not valid outside of $\mpt$ in general, as the next result shows.
\begin{Prop}\label{prop:Enon}
The set $E$ does not have the RNP.
\end{Prop}
\begin{proof}
Formally, we would like to say that $\rho \mapsto \mcl T_\rho$ defined in Proposition \ref{prop:exT} is an isomorphic embedding of $L^1$ into $E$. 
Since $L^1$ does not have the RNP, neither does $E$. This type of reasoning is quite used in the functional analysis literature (see \citep{aliaga2026}) but this route does not look quite correct to us here. 

Instead, we provide a more direct argument which consists in constructing a bounded $2^d$ regular $\delta$-tree inside $E$. Recall that, a $k$-regular tree is a sequence $(x_{i})_{i \geq 1}$ such that for all $i \geq 1$, $x_{i} = \frac1k (x_{ki} + x_{ki+1} +\dots + x_{ki + k-1})$. It is a $\delta$ tree if for all $i$, $0 \leq j \leq k-1$, $\|x_i -x_{ki + j}\|\geq \delta$. Such a set is clearly not dentable because every one of its point is the convex combination of points at distance more than $\delta$.

We consider the unit cube $[0,1]^d$ of $\R^d$ a regular subdivision $(Q_i)_{i \geq 1}$ of it. For $i \geq 1$, define $\psi_i = \frac{1}{|Q_i|}\mathbb 1_{Q_i}$. The family $(\psi_i)_{i \geq 1}$ is a $2^d$ regular tree in $L^1$, thus $(T_{\psi_i})_{i \geq 1}$ is $2^d$ regular tree in $E$. We consider $i \geq 1, 0 \leq j \leq k-1$ and we want to bound from below $\|T_{\psi_i}-T_{\psi_{ki + j}}\|_E$. 

In order to do so, we consider two disjoint balls $B_1$ and $B_2$, included in respectively $Q_{ki + j}$ and $Q_i \setminus Q_{ki + j}$, such that their radii are equal to a fourth of the side of $Q_{ki+j}$.
Let $g$ be defined by $1$ in $B_1$, $-1$ in $B_2$ and $0$ elsewhere. Let $f$ be defined by $\Delta f = g$. We have $f \in X$ and the exact computation of the solution of $\Delta f = \mathbb 1_B$ for a ball $B$ yields that for some $C$ depending only on $d$, $ \|f\|_X \leq C$. Now observe that 
$$
\ba
\|f\|_X\|T_{\psi_i}-T_{\psi_{ki + j}}\|_E &\geq \int_{\R^d}( \psi_i -  \psi_{ki +j})g = \int_{B_1}\frac{1}{|Q_{ki + j}|} - \frac{1}{|Q_{i}|} - \int_{B_2}\frac{1}{|Q_{i}|}.
\ea
$$
The right hand side is bounded by a constant independent of $i$ and $j$, thus $(T_{\psi_i})_{i \geq 1}$ is a $\delta$-tree for some $\delta > 0$. Since for all $i \geq 1, \phi \in X$, $T_{\psi_i}(\phi) \leq \|\phi\|_X$, it follows that this $\delta$-tree is also bounded. Since it is not dentable, it follows that $E$ fails the RNP.
%
\end{proof}

We conclude this section with the following remark.
\begin{Rem}\label{rem:extension}
The regularity restriction in the definition of $X$ could have been strengthened, and the content of this section would have been left unchanged (up to some constants). Indeed, for $k \geq 2$, we could have set $X^k := W^{k,\infty}\cap W^{2,\infty}$. In such a setting, $E$ is replaced by $E^k$, a predual of $X^k$ which we can construct in exactly the same way, and all the results above are immediately adaptable.
\end{Rem}

\section{Perturbed optimization}\label{sec:optim}
We now turn to one of the main consequences of Theorem \ref{thm:1}: a perturbed optimization result on $\mpt$, sometimes also called a variational principle. Perturbed optimization results have a long history, notably because of their importance in the calculus of variation to create points of maximum/minimum. A fundamental result in this field is of course Ekeland's variational principle, which is concerned with Lipschitz perturbations of functions to obtain existence of almost minima. This result is very general and valid in complete metric spaces. In \citep{BP}, Borwein and Preiss obtained a smooth variational principle on certain Banach spaces (smooth meaning that the perturbation is smooth). Here, we produce a variational principle in the spirit of Stegall's Lemma, see \citep{stegall} and Theorem 5.15 in \citep{phelps}, which states that on a Banach space with the RNP, there is a $G_\delta$ dense set of linear perturbations which create strongly exposed points of minimum of bounded lsc functions on bounded sets. On Wasserstein spaces, variational principles were also obtained in \citep{cosso2021master,bayraktar}, and we believe our result to be more general. See also the work of Bertucci and Lions in \citep{nonstegall} which is also concerned with Wasserstein spaces and extensively commented in Part \ref{part:hor}.\\

The following result holds.
\begin{Theorem}\label{thm:stegall}
Let $\mathscr C$ be a closed bounded convex subset of $\mpt$ and let $f: \mathscr C \mapsto \R$ be lsc and bounded from below. Then, for any $k \geq 2$, there is a $G_\delta$ dense subset $A$ of $X^k$ such that for any $\phi \in A$, $µ \mapsto f(µ) + \langle \phi,µ\rangle$ has a strongly exposed point of minimum on $\mathscr C$.
\end{Theorem}
\begin{proof}
Propositions \ref{prop:dent} and Theorem 5.15 in \citep{phelps} yields the result in the case $k =2$. In the case $k > 2$, we saw in Remark \ref{rem:extension} that Proposition \ref{prop:dent} still holds in $E^k$ so the result also holds in this case.
\end{proof}

Another standard conclusion that we can draw from Theorem \ref{thm:1} is the following stronger form of Krein-Milman Theorem.
\begin{Theorem}
Let $\mathscr C$ be a closed bounded convex subset of $\mpt$, then for any $k \geq 2$, $\mathscr C$ is the closed convex hull of its strongly exposed points by linear functionals of the form $µ \mapsto \langle \phi, µ\rangle$ for $\phi \in X^k$.
\end{Theorem}
\begin{proof}
See Theorem 5.20 in \citep{phelps}.
\end{proof}

\section{Radon-Nikodym theorems on $\mpt$}\label{sec:analytic}
In this section, we explain how to obtain a "usual" form of the Radon-Nikodym Theorem on $\mpt$, from Theorem \ref{thm:1}. Before going into the details of vector valued measures, let us insist upon the fact that the link between such theorems and the dentability of bounded sets is already very well understood, see \citep{diestel}. However, contrary to the previous section on perturbed optimization, we were not able to find such results in cases in which the dentability holds only on a certain subset of the Banach space in question, that is why we shall give more details on the proof, even though they are entirely classical. It is also why these proofs are deferred to the Appendix.\\

Let $(\Omega, \mcl A)$ be a measurable space, and $µ$ a finite, non-negative, $\sigma$ additive measure on it. Let $Y$ be a Banach space. A map $G : \mathcal A \mapsto Y$ is a vector measure on $Y$ provided that for any pair of disjoint sets $A_1,A_2$ in $\mcl A$, $G(A_1 \cup A_2) = G(A_1) + G(A_2)$. Such a map is said to be $µ$-continuous if for all $A \in \mcl A$, such that $µ(A) = 0$, it holds that $G(A) = 0$.

A vector measure $G$ is of bounded variation if 
$$
\sup_{\pi \subset \mcl A} \sum_{A \in \pi} \|G(A)\| < \infty,
$$
where the supremum is taken over all countable partitions of $\Omega$.

A function $f : \Omega \mapsto Y$ is called $µ$-measurable if there exists a sequence of simple functions $f_n : (\Omega,\mcl A) \mapsto Y$ such that $\lim_{n \to \infty} \|f_n -f\| = 0$ $µ$-almost everywhere. We denote by $L^1(µ,Y)$ the set of $µ$-measurable functions $f$ such that 
$$
\int_\Omega \|f\| dµ < \infty.
$$
Because of Proposition \ref{prop:Enon}, we do not have a Radon-Nidokym Theorem on $E$ (where $E$ is defined in Section \ref{sec:fa}). On the other hand, there is no vector measure valued in $\mpt$ since $0_E \notin \mpt$. Hence, we concentrate our attention on the cone $\mcl K$ of $E$ generated by $\mpt$. We have the following.
\begin{Prop}
Any bounded $\mathscr C \subset\mcl K$ is dentable in $E$.
\end{Prop}
\begin{proof}
The fact that $\mathscr C$ is bounded in $E$ implies that $\{\int_{\R^d}dµ | µ \in \mathscr C\}$ and $\{M_2(µ) | µ \in \mathscr C\}$ are bounded. Thus Prokhorov's Theorem still applies and all the arguments used in the proof of Theorem \ref{thm:1} remain valid, once one changes the $W_2$ topology by the topology induced by weak convergence plus convergence of the second moments in the set of non-negative measures.
\end{proof}
Thus, $\mcl K$ satisfies the RNP.

\begin{Theorem}\label{thm:RNP}
Let $G: \mcl A \mapsto \mcl K\subset E$ be a $µ$-continuous vector measure of bounded variation. Then, there exists $f \in L^1(µ,\mcl K)$ such that for any $A \in \mcl A$,
$$
G(A) = \int_A f(\omega) µ(d\omega).
$$
\end{Theorem}

\section{Martingale convergence theorems}\label{sec:mart1}
In this section, we show, quite similarly to what has been done in the previous section, that the RNP property of Theorem \ref{thm:1} implies a mean convergence theorem for $\mpt$ valued martingales. Martingales are ubiquitous in probability theory, and we are not going to give a detailed presentation of them, or of the importance of their convergence results. We refer to \citep{pisier} for a general presentation of martingales in Banach spaces.

Let $(\Omega, \mcl A, \mathbb P)$ be a probability space. Let $T \subset \R$ be unbounded from above. We write $\lim_t$ for the limit $t \to \infty$ along elements in $T$. Consider a monotone non-decreasing family of sub-$\sigma$-algebra $(B_t)_{t \in T}$ of $\mcl A$, i.e. for any $t_1,t_2 \in T$ such that $t_1 \leq t_2, B_{t_1} \subset B_{t_2}$, and $B_{t_1}$ is a sub-$\sigma$-algebra of $\mcl A$. 

For $B$ a sub-$\sigma$-algebra of $\mcl A$, we denote by $\mathbb E[\cdot|B]$ the conditional expectation operator defined on $L^p(\mathbb P,X)$ for any $p \in [1,\infty)$ and $X$ Banach space, where $L^p(\mathbb P,X)$ is the $L^p$ space of Bochner integrable $X$ valued functions over $\Omega$ (see Theorem V.1.4 in \citep{diestel}).

By definition, for $X$ a Banach space, an $X$-valued martingale $(f_t)_{t \in T}$ is a family in $L^1(\mathbb P,X)$ such that for all $t_1,t_2 \in T, t_1 \leq t_2$,
$$
\mathbb E[f_{t_2}|B_{t_1}] = f_{t_1}.
$$
Such a martingale is said to be $C$ valued if for all $t \in T$, $f_t$ is valued in $C$. We can obtain, as a consequence of Theorem \ref{thm:RNP}, the following result.
\begin{Theorem}\label{thm:charles}
Let $(f_t)_{t \in T}$ be a $\mpt$ valued martingale in $L^1(\mathbb P,E)$ such that
\begin{itemize}
\item $\sup_{t \in T} \|f_t\|_1 < \infty,$
\item $\lim_{\mathbb P(A) \to 0, A \in B_t} \sup_{s \geq t} \int_A \|f_s\|\,d\mathbb P = 0.$
\end{itemize}
Then, there exists $f \in L^1(\mathbb P,\mpt)$ such that $\lim_{t } \|f_t - f\|_1 = 0.$
\end{Theorem}
\begin{proof}
The proof follows once again the one from \citep{diestel}. For $A \in \cup_{t \in T} B_t$, we set $F(A) =\lim_{t } \int_A f_t d\mathbb P$. Note that such a limit indeed exists for all such $A$ since the sequence is stationary after a certain time. Furthermore, because $\mathbb P \geq 0$, $F$ is valued in $\mcl K$, the cone generated by $\mpt$ in $E$. It is rather immediate to check that $F$ is of bounded variation, and the uniform integrability implies that it is $\mathbb P$-continuous. We thus use Theorem I.5.2 in \citep{diestel} to consider a vector measure $G$, defined on $\mcl A_0$ the $\sigma$-algebra generated by $\cup_{t \in T} B_t$, which extends $F$. It is of bounded variation and is $\mathbb P$-continuous as well. By the way it is constructed, $G$ is valued in $\ol{co}(F(\cup_{t \in T} B_t)) \subset \mcl K$. Thus, from Theorem \ref{thm:RNP}, $G$ has a Radon-Nikodym derivative $f \in L^1((\Omega, \mcl A_0,\mathbb P),E)$. Then, for all $A \in \cup_{t \in T} B_t$, 
$$
\lim_{t} \int_Af_t d\mathbb P = F(A) = G(A) = \int_Af\,d\mathbb P.
$$
One has $\mathbb E[f|B_t] = f_t$ for all $t$ and $F(A) = \int_A f d\mathbb P$ for all $A \in \cup_{t \in T} B_t$. We want to show that $\lim_t \|f_t - f\|_1 = 0$, which implies the result since $f$ is $\mcl A $ measurable. We now use the fact that $f$ can be approximated in $L^1$ by simple functions to obtain a family $(g_\eps)_{\eps > 0}$, such that for all $\eps > 0$, $g_\eps$ is a simple function on $(\Omega,\mcl A_0,\mathbb P)$ and $\|g_\eps - f\|_1 \leq \eps$. Because $g_\eps$ is a simple function, up to changing slightly $g_\eps$, it can be written as a simple function of a finite number of sets of $\cup_{t \in T} B_t$. In particular, for some $t_0$ depending on $\eps$, it is $B_{t_0}$ measurable. It follows that for $t\geq t_0$
$$
\|f_t-f\|_1 \leq \|f_t - g_\eps\|_1 + \|g_\eps - f\|_1 \leq \|\mathbb E[f-g_\eps|B_{t_0}]\|_1 + \eps \leq 2\eps.
$$
Hence the result follows provided that $f$ is $\mpt$ valued. But, because all the $f_t$ are $\mpt$ valued, the convergence implies that $f$ is necessarily $\mpt$ valued ($\mathbb P$ almost surely) as well. 
\end{proof}

\section{Extensions}\label{sec:extension}
The results above are in fact valid in much more general spaces than $\mpt$. The extension of the results above to the cases we are about to mention are routine checks of the proof we presented that we do not detail. We simply list several extensions.

\subsection*{Extension to $\mpprd$}
\bigskip
First, for any $1 \leq p < \infty$, $\mpprd$ satisfies the RNP, up to the appropriate change of the functional space $E$, notably to take into account the integrability conditions of the measures. For instance, if $p=1$, the set $X$ can be taken as the set of Lipschitz functions with suitable norm, (or smoother functions with at most linear growth). If $p > 2$, then the Hessian of the functions in $X$ cannot be bounded because $x \mapsto |x|^p$ has to be an element of $X$. Bounding the derivatives of order $k$ for $k$ such that $k \geq p$ and changing the first term in the definition of $\|\phi\|_X$ by $\sup_{x \in \R^d} \{\frac{|\phi(x)|}{1 + |x|^p}\}$ is convenient for instance. So is requiring that the Hessian is not bounded but is allowed only a certain growth.

The Radon-Nikodym Theorem and the convergence of uniformly integrable martingales on $\mpprd$ then holds by similar arguments.

All those possible generalizations, in our opinion, argue in favour of stating the RNP property for Wasserstein spaces as follows.
\begin{Theorem}
Let $p \in [1,\infty)$, $\mathscr C$ be a bounded set of $\mpprd$. Then,
\begin{itemize}
\item there exists $\ol µ \in \mathscr C$ such that for all $\eps > 0$, $\ol µ\notin \ol{co}(\mathscr C \setminus B(\ol µ,\eps))$,
\item for any $k \geq 1$, for any $\eps > 0$, there exist $\kappa > 0$ and a $\mcl C^k$ function $\phi : \R^d \mapsto \R$ which grows at most like a multiple of $x \mapsto |x|^p$ such that  
$$
\left\{µ \in \mathscr C \bigg| \langle \phi,µ\rangle \leq \inf_{\nu \in \mathscr C} \langle \phi,\nu\rangle + \kappa\right\}
$$
is non empty with diameter at most $\eps$.
\end{itemize}
\end{Theorem}
From the previous, it is then classical to obtain the following perturbed optimization result.
\begin{Theorem}\label{thm:mainstegall}
Let $p \in [1,\infty)$, $\mathscr C$ be a bounded closed and convex set of $\mpprd$. Let $f: \mathscr C \mapsto \R$ be a bounded from below and lsc function. Then, for any $k \geq 0$, $\eps > 0$, there exists $\phi \in \mcl C^k$ such that 
\begin{itemize}
\item
$$
\left\| \frac{\phi}{1+|x|^p}\right\|_\infty + \|D^k\phi\|_\infty \leq \eps,
$$
\item the map $µ \mapsto f(µ) + \langle \phi,µ\rangle $ has a strongly exposed point of minimum on $\mathscr C$.
\end{itemize}
\end{Theorem}

\subsection*{Other spaces than $\R^d$}
The precise nature of $\R^d$ as a base space for the Wasserstein space $\mpt$ was only scarcely used in the previous analysis. Notably, it only came into play when using the following properties of functional spaces defined on $\R^d$: notions of smooth functions can be defined, Sobolev spaces with integrability conditions $p = 2$ are Hilbert spaces, usual functional spaces $\mathcal C^k$ are Banach spaces when equipped with norms which bound (and measure) the growth of the functions. Such results are true in much more general settings, and we do not enter into the question of stating the most general base space possible for our Wasserstein space. We simply note that smooth manifolds with bounded geometry could have been treated by using exactly the same arguments and also that any closure of a smooth domain of $\R^d$ would have raised no issue. In particular, non-connected domains are treated in the same manner. We insist upon this last extension as it shows that the current (vertical) geometry does not see a lot the geometry of the base space, which is a huge difference compared with the horizontal geometry studied in the next part.

\bigskip
\bigskip
\part{Horizontal geometry}\label{part:hor}
In this part, we consider only the case in which the Wasserstein spaces are defined on $\R^d$. We only indicate here that, for $p \in [1,\infty)$, and $\mo$ the closure of a convex domain of $\R^d$, $\mppo$ is a closed coupling convex subset of $\mpprd$, and thus suitable for adaptation of most of what follows.\\

The main difficulty that we are facing in this part is that we will not be able to use a nice linear structure that our Wasserstein spaces would inherit from some larger Banach spaces, as we did above. We shall work without a linear structure, which makes less natural the definitions of strongly exposed points, of regularity, or of martingale. Also, since we work without a linear structure, defining $\mpprd$-valued vector measures is not natural, so we shall not address the question of analytic forms of the Radon-Nikodym Theorem.\\

Recall that the space $\mathbb L^p$ defined above satisfies the RNP if and only if $p \in (1,\infty)$ and that the notion of coupling convexity is defined in the Preliminaries. We shall argue in this part by lifting $\mpprd$ to $\mathbb L^p$ and by pulling down properties of the linear lifted structure to the Wasserstein spaces. This idea of Lions was also notably used for perturbed optimization by Bertucci and Lions \citep{nonstegall} and for defining a tangent space by Bertucci \citep{bertucci2025tangent}.

\section{Perturbed optimization}
The main idea of this section is the following: $\mpprd$ can be viewed as a quotient of $\mathbb L^p$ for the equivalence relation: "having the same law", and since for $p\in (1,\infty)$ $\mathbb L^p$ has the RNP, we can leverage it at the level of $\mpprd$.\\

In \citep{nonstegall}, Bertucci and Lions established the following non-linear version of Stegall's Lemma, with the exception of the fact that the set $A$ was only proven to be non-empty.
\begin{Theorem}\label{thm:nonstegall}
Let $p \in (1,\infty)$, $\scr C\subset \mpprd$ be a coupling convex, closed and bounded subset of $\mpprd$. Let $f: \scr C \mapsto \R$ be lsc and bounded from below. Then, there exists a $G_\delta$-dense set $A$ of $ \mcl P_{p'}(\R^d)$ such that for $\nu \in A$, $µ \mapsto f(µ) + \mcl I_p(µ,\nu)$ has a strongly exposed point of minimum in $\mathscr C$.
\end{Theorem}
\begin{proof}
We reproduce the proof of \citep{nonstegall} for the sake of completeness. Write $\mo := \{X \in \mathbb L^p, \mathcal L(X) \in \scr C\}$ and set $F: \mo \mapsto \R$ defined by $F(X) = f(\mathcal L(X))$. Because $\scr C$ is coupling convex, closed and bounded, we deduce that $\mo$ is a convex, closed and bounded set of $\mathbb L^p$. Hence, from the usual version of Stegall's Lemma (Theorem 5.15 in \citep{phelps}), there is a $G_\delta$ dense set $B$ of $\mathbb L^{p'}$ such that for all $Y \in B$, $X \mapsto F(X) - \langle Y,X\rangle$ has a strongly exposed point of minimum on $\mo$. Fix such a $Y$, consider $X^*$ the associated point of strongly exposed minimum and write $\nu = \mathcal L(Y)$, $µ^* = \mcl L(X^*)$. A simple check yields that $µ \mapsto f(µ) + \mcl I_p(µ,\nu)$ has a strongly exposed point of minimum on $\scr C$ at $µ^*$.

The proof is thus complete if we are able to show that there is a $G_\delta$ dense set of such elements of $\mcl P_{p'}(\R^d)$. We place ourselves at the level of the lift and consider the set 
$$
B_k := \left\{Y\in \mathbb L^{p'} \bigg| \exists \alpha \in \R, diam\left( X \in \mo, F(X) - \langle Y,X\rangle \leq \inf_{\mo}\{F(\cdot) - \langle Y,\cdot\rangle\} + \alpha\right) \leq \frac 1k\right\},
$$
with the convention that the diameter of the empty set is $+ \infty$. From Stegall's Lemma in $\mathbb L^p$, each $B_k$ is dense, furthermore an immediate check yields that they are also all open. We want to show that $A:= \cap_{k \geq 1} \mcl L(B_k)$ is the $G_\delta$ dense set claimed in the statement.


Consider an open set $Q$ of $\mathbb L^{p'}$ and define $R := \{\mcl L(X)| X \in Q\}$. We want to show that $R$ is open as well. Let $µ \in R$, there exists $X \in Q$ such that $\mathcal L(X) = µ$. Since $Q$ is open, there exists $ \eps > 0$ such that $\|Y - X \|\leq \eps \Rightarrow Y \in Q$. Let $\nu \in \mcl P_{p'}(\R^d)$ such that $W_{p'}(\mu,\nu) \leq \frac\eps2$. Let $(\tilde X, \tilde Y)$ be an optimal coupling for $W_{p'}(µ,\nu)$. Using Lemma 5.23 in \citep{carmona2018probabilistic}, we know that there exists a measure preserving mapping $\tau: \tilde \Omega \mapsto \tilde \Omega$ such that $\|\tilde X \circ \tau - X\|_{\infty} \leq \frac\eps2$. We now estimate 
$$
\|X- \tilde Y \circ \tau\|_{p'} \leq \|X - \tilde X \circ \tau\|_{p'} + \|\tilde X - \tilde Y \|_{p'} \leq \eps.
$$
Hence $\tilde Y \circ\tau \in Q$, from which $\nu \in R$ follows. 

Furthermore each $\mcl L(B_k)$ is dense because so is $B_k$. Hence $A$ is indeed a $G_\delta$ dense set because of Baire's Theorem. It remains to verify that the result is true on the points of $A$. Let $\nu \in A$, $k \geq 1$, $Y_k \in B_k$ such that $Y_k \sim \nu$ (and $\alpha$ associated through the definition of $B_k$) and $µ_1,µ_2$ such that
$$
 f(µ_i) + \mcl I_p(µ_i,\nu)  \leq \inf_{µ'}f(µ') + \mcl I_p(µ',\nu) + \alpha.
$$
Up to a use of Lemma 5.23 in \citep{carmona2018probabilistic} that we do not detail, we can consider $X_i$ such that $X_i \sim µ_i$ and $(X_i,Y_k)$ is optimal for $\mcl I_p(µ_i,\nu)$. It follows that $W_p(µ_1,µ_2) \leq \frac1k$. Hence all minimizing sequences are Cauchy and have the same limit. Thus $\nu$ actually strongly exposes a point of minimum and the result is proved.
\end{proof}

The previous perturbed optimization result does not use linear perturbations as Stegall's Lemma. Nonetheless, it will be sufficient to deduce usual geometric properties of spaces with the RNP.

\section{Coupling convex sets and functions}\label{sec:geohor}
In this section, we present some results about coupling convex sets and functions. Notably, we show that Krein-Milman's like theorems are available in this geometry, prove that coupling convex functions are differentiable in many points (similar to a sort of Asplund space property), and we recall a duality result of Pinzi and Savar\'e \citep{pinzi}.

The differentiability result we provide for coupling convex functions (Theorem \ref{thm:asplund}), could be reformulated as the fact that, in this horizontal geometry, Wasserstein spaces are also Asplund spaces. This is natural because, at the level of the lift, $\mathbb L^p$ is an Asplund space since it is reflexive.

\subsection{Extreme points of coupling convex sets}
In addition to Theorem \ref{thm:nonstegall}, using that for any $\nu \in \mcl P_{p'}(\R^d)$, $µ \mapsto \mcl I_p(µ,\nu)$ is continuous and coupling concave yields the following result.
\begin{Theorem}\label{thm:ccdent}
Let $p \in (1,\infty)$ and $\scr C$ be a coupling convex, closed and bounded subset of $\mpprd$. There exists $\ol µ \in \scr C$ such that for any $\eps > 0$, $\ol µ \notin \ol{cco}(\scr C \setminus B(\ol µ,\eps))$.
\end{Theorem}
\begin{proof}
We apply Theorem \ref{thm:nonstegall} to $f \equiv 0$ and deduce the existence of $\nu \in \mcl P_{p'}(\R^d)$ and $\ol µ \in \scr C$ such that $\ol µ$ is strongly exposed by $\mcl I_p(\cdot,\nu)$. In particular, for all $\eps > 0$, there exists $\delta > 0$ such that 
$$
\inf \left \{ \mcl I_p(µ,\nu) | µ \in \scr C \setminus B(\ol µ,\eps) \right\} \geq \mcl I_p(\ol µ,\nu) + \delta.
$$
Because $\mcl I_p(\cdot,\nu)$ is coupling concave and continuous, we deduce that for any $µ \in \ol{cco}(\scr C \setminus B(\ol µ,\eps))$, $\mcl I_p(µ,\nu) \geq \mcl I_p(\ol µ,\nu) + \delta$, hence the result.
\end{proof}

We can also obtain an equivalent of Krein-Milman Theorem in this setting.
\begin{Theorem}
Let $p\in (1,\infty)$ and $\scr C$ be a coupling convex, closed and bounded subset of $\mpprd$. Then $\scr C$ is closure of the coupling convex hull of its strongly exposed points by functionals of the form $\mcl I_p(\cdot,\nu)$ for $\nu \in \mcl P_{p'}(\R^d)$.
\end{Theorem}
\begin{proof}
We proceed as in the proof of Theorem \ref{thm:nonstegall}. Let $\mo := \{ X \in \mathbb L^p | \mcl L(X) \in \scr C\}$. Let $B:= \{ X \in \mo | \exists Y \in \mathbb L^{p'}, X' \mapsto \langle Y,X'\rangle \text{ has a strongly exposed minimum at } X\}$. Since $\mathbb L^p$ has the RNP, we know that $\mo = \ol{co}(B)$. We saw in the proof of Theorem \ref{thm:nonstegall} that for any $X \in B$, $\mcl L(X)$ is strongly exposed by a functional $ \mcl I_p(\cdot,\nu)$ for some $\nu \in \mcl P_{p'}(\R^d)$. Hence the result follows since for any $µ \in \scr C$, there exists $X \in \mo$ of law $µ$.
\end{proof}

\subsection{Continuous coupling convex functions}
In this section, we state two results on continuous coupling convex functions on $\mpprd$, for $p \in (1,\infty)$, that are apparent from the previous study, and which we believe can have an interest outside of this work. The first one states that continuous coupling convex functions are differentiable in a lot of points, and the second one that the family of coupling convex functions $(-\mcl I_p(\cdot,\nu))_{\nu \in \mcl P_{p'}(\R^d)}$ generates all the continuous coupling convex functions of $\mpprd$. The second result is also established in \citep{pinzi}, toward which we refer for additional results on coupling convex functions. 
We recall that a function $U : \mpprd \mapsto \R$ is horizontally differentiable at $µ \in \mpprd$ if there exists $\phi \in L^{p'}(\R^d,\R^d)$ such that for any $\nu \in \mpprd, \gamma \in \Pi(µ,\nu)$
$$
\lim_{\int |x-y|^pd\gamma \to 0} \frac{U(\nu) - U(µ) - \int_{\R^{2d}}\phi(x)\cdot(y-x)\gamma(dx,dy)}{\left(\int_{\R^{2d}} |x-y|^p\gamma(dx,dy)\right)^\frac1p} = 0.
$$

\begin{Theorem}\label{thm:asplund}
Let $p \in (1,\infty)$, $U: \mpprd\mapsto \R$ be continuous and coupling convex. Then, it is (horizontally) differentiable on a dense $G_\delta$ set.
\end{Theorem}
\begin{proof}
Let $\mcl U: \mathbb L^p \mapsto \R$ be defined by $\mcl U(X) = U(\mcl L(X))$. Then, $\mcl U$ is continuous and convex, on $\mathbb L^p$ which is an Asplund space since it is reflexive (recall that $p \in (1,\infty)$). Recalling Propositions 1.23 and 1.25 in \citep{phelps}, we know that 
\begin{itemize}
\item$$
B_k := \bigcup_{t  >0}\left\{ X \in \mathbb L^p | \Delta_t \mcl U(X) \leq \frac1k\right\}
$$
is open, where 
$$
\Delta_t \mcl U(X) := \frac1t\sup_{\|H\|\leq 1} \{\mcl U(X+ tH) + \mcl U(X-tH) - 2\mcl U(X)\}.
$$
Note that for any $X \in \mathbb L^p$, $\Delta_t \mcl U(X)$ is non-decreasing in $t$. 
\item The $G_\delta$ set $\cap_{k \geq 1} B_k$ is the set of Fr\'echet differentiability of $\mcl U$.
\end{itemize}
Since $\mathbb L^p$ is Asplund, the previous intersection is dense, and thus each of the $B_k$ is dense. Because $\Delta_t U$ is Lipschitz continuous in $X$ (with a constant depending on $t$ of course), we deduce that $\Delta_t \mcl U$ is in fact law invariant, and thus that if $X \in B_k$, then for any $X'$ such that $\mcl L(X') = \mcl L(X)$, $X' \in B_k$. We now set $A_k = \mcl L(B_k)$. Arguing as in the proof of Theorem \ref{thm:nonstegall}, we know that $A_k$ is open. It is also dense. Thus $\cap_{k \geq 1} A_k$ is a dense $G_\delta$ set of $\mpprd$. 

Let $µ \in \cap_{k\geq 1} A_k$ and $X \sim µ$. We deduce from what we have established above that $X \in \cap_{k \geq 1} B_k$, thus $X$ is a point of Fr\'echet differentiability of $\mcl U$. From \citep{gangbotudorascu}, \citep{alfonsi} in the case $p=2$ or Section 4 in \citep{bertuccibook} in the general case, we deduce that $U$ is then horizontally differentiable at $µ$ and the result follows.
\end{proof}
\begin{Rem}
Note that the previous theorem applies directly to $\mcl I_p$. Since the differentiability of such maps is directly linked to the uniqueness of optimal transport plans between its arguments, we have also proven that for any measure $\nu \in \mpt$, there is a dense $G_\delta$ subset of $\mpt$ for which there is a unique optimal transport plan toward $\nu$.
\end{Rem}

\begin{Theorem}\label{thm:repconvex}
Let $U: \mpprd \mapsto \R$ be continuous and coupling convex. Consider the function $U^*: \mcl P_{p'}(\R^d) \mapsto \R\cup\{+ \infty\}$ defined by 
$$
U^*(\nu) = \sup_{µ \in \mpprd} \{-\mcl I_p(µ,\nu) - U(µ)\}
$$
Then the following holds for any $µ \in \mpprd$
$$
U(µ) = \sup_{\nu \in \mcl P_{p'}(\R^d)}\{-\mcl I_p(µ,\nu) - U^*(\nu)\}.
$$
\end{Theorem}
Note that for any $\nu \in \mcl P_{p'}(\R^d)$, $U^*(\nu)$ is the biggest constant $c \in \R\cup\{+\infty\}$ such that $µ \mapsto -\mcl I_p(µ,\nu) - c$ is lower than $U$ everywhere.
\begin{proof}
By construction it follows that for any $µ \in \mpprd$
$$
U(µ) \geq \sup_{\nu \in \mcl P_{p'}(\R^d)}\{-\mcl I_p(µ,\nu) - U^*(\nu)\}.
$$
Now let $µ^* \in \mpprd$, observe that thanks to Section 5 in \citep{bertuccibook}, because $U$ is continuous and coupling convex, $U$ is (horizontally) sub-differentiable at $µ^*$, i.e. there exists $\gamma \in \mcl P(\R^d\times \R^d)$ such that $(\pi_1)_\#\gamma = µ^*, \nu^* := (\pi_2)_\#\gamma \in \mcl P_{p'}(\R^d)$ such that for any $µ \in \mpprd$, $\Gamma \in \Pi(\gamma,µ)$
\be\label{eq:subdiff}
U(µ) \geq \int_{(\R^d)^3} z\cdot (y-x) \Gamma(dx,dz,dy) + U(µ^*).
\ee
Evaluating the previous for any $\Gamma \in \Pi(\gamma,µ^*)$, we obtain that $\gamma$ is optimal in $\mcl I_p(µ^*,\nu^*)$. This optimality, once used in \eqref{eq:subdiff} implies that 
$$
µ\mapsto - \mcl I_p(µ,\nu^*)  + \mcl I_p(µ^*,\nu^*) + U(µ^*) 
$$
is bounded from above by $U$ and that the equality holds at $µ^*$, hence the result follows.
\end{proof}

\section{Horizontal martingales}\label{sec:marthor}
\subsection{Definition and convergence theorem}
In this section, we define horizontal martingales on Wasserstein spaces and prove a convergence result for them. The definition we are about to give is in agreement with the horizontal geometry of the Wasserstein spaces. Defining a proper conditional expectation on a general metric space does not seem to be a feasible task and we do not address it here. Instead, we rely heavily on the horizontal geometry. In particular we do not pass through the barycenters to define martingales, such as it has been done by Sturm \citep{sturm}, even though there are the Wasserstein barycenters of Agueh and Carlier \citep{carlier}. The main reason why is that Sturm's theory is mainly developed for non-positive curvature space, which the Wasserstein spaces are not.

 Defining martingales on non-linear spaces is a question which dates back to the early developments of stochastic calculus in the beginning of the second half of the last century, notably to work with stochastic processes valued in manifolds. Having this literature in mind (which we recall next), the following definition should not be surprising.

\begin{Def}
Let $p\in (1,\infty)$, $T \subset \R$, $(\Omega,\mcl A,\mathbb P)$ be a probability space and consider an increasing sequence $(B_t)_{t \in T}$ of sub-$\sigma$-algebras of $\mcl A$. The $\mpprd$ valued and adapted process $(µ_t)_{ t \in T}$ is a $p$-horizontal martingale if:
\begin{itemize}
\item for all $t \in T$, $\mathbb E\left[(M_p(µ_t))^\frac1p\right] < \infty$,
\item for all continuous and coupling convex function $\phi: \mpprd \mapsto \R$, the process $(\phi(µ_t))_{t \in T}$ is a sub-martingale.
\end{itemize}
\end{Def}
\begin{Rem}
Given a Banach space $E$, a usual martingale $(X_t)_{t \in T}$ and a convex function $\phi: E \mapsto \R$, Jensen's inequality implies that $(\phi(X_t))_{t \in T}$ is indeed a sub-martingale. We leave to the interested reader the verification that it is in fact a characterization in this linear setting.
\end{Rem}
\begin{Rem}
In the previous definition, the integrability condition on $µ_t$ guarantees that for all continuous coupling convex function $\phi$, $\mathbb E[\phi(µ_t)]$ is well defined in $(-\infty,\infty]$. Indeed, if we look at the lift $V$ of $\phi$ on $\mathbb L^p$, which is convex and continuous on $\mathbb L^p$, we know that for $X_t \sim µ_t$, $\mathbb E[V(X_t)]$ is well defined as soon as $X_t$ has a first moment. 
\end{Rem}
The (fundamental) idea to define a semi-martingale on a manifold through its image by smooth real functions defined on the manifold dates back to Schwartz \citep{schwartz}. When concerned with defining martingales valued in manifolds through convex functions, it is difficult to pinpoint a single author. In \citep{darling}, Darling gave a definition and credited Eells and Elworthy, while in \citep{sanslarmes}, Meyer credited Bismut for this definition.

The existence of such martingales is guaranteed by the following result.
\begin{Prop}
Let $p\in (1,\infty)$ and $(X_t)_{t \in T}$ be a $\mathbb L^p$ valued martingale. The process $((X_t)_\#Leb)_{t \in T}$ is a $p$-horizontal martingale.
\end{Prop}
\begin{proof}
Let $\phi: \mpprd\mapsto \R$ be continuous and coupling convex. Then, $\psi: \mathbb L^p\mapsto \R$ defined by $\psi(X) = \phi(\mcl L(X))$ is continuous and convex, see Section 5 in \citep{bertuccibook}. From, Jensen's inequality, $(\psi(X_t))_{t \in T}$ is a sub-martingale, thus the result follows.
\end{proof}
In fact, there is a form of converse to the previous result, at least for discrete martingales. We are able to prove the following.
\begin{Theorem}\label{thm:liftmart}
Let $p \in (1,\infty)$ and $(µ_n)_{n \geq 0}$ be a $p$-horizontal martingale. Then, there exists a $\mathbb L^p$ valued martingale $(X_n)_{n \geq 0}$ such that for any $n \geq 0$, almost surely, $(X_n)_\#Leb =µ_n$.
\end{Theorem}
Before proving this result, we note that from it, we can deduce the following result of convergence of horizontal martingales, stated for bounded horizontal martingales (weaker integrability conditions are of course possible, following standard arguments that we do not present here).
\begin{Theorem}\label{thm:martconv}
Let $p \in (1,\infty)$ and $(µ_n)_{n \geq 0}$ be a $p$-horizontal martingale which satisfies: $\exists K > 0, \forall n \geq 0, M_p(µ_n) \leq K$ almost surely.
Then there exists a $\mpprd$ valued random variable $µ_\infty$ such that, for all $k \geq 1$, $\lim_{n \to \infty}\mathbb E[W^k_p(µ_n,µ_\infty)]= 0$.
\end{Theorem}
\begin{proof}
Using Theorem \ref{thm:liftmart}, we deduce that there exists an associated $\mathbb L^p$ valued martingale $(X_n)_{n \geq 0}$. Since, for $p \in (1,\infty)$, $\mathbb L^p$ has the RNP (it is reflexive), it follows that $(X_n)_{n \geq 0}$ converges in $L^k(\mathbb P,\mathbb L^p)$ toward some limit $X_\infty\in  L^\infty(\mathbb P, \mathbb L^p)$. Set $µ_\infty = (X_\infty)_\#Leb$ and remark that, almost surely, $W_p(µ_n,µ_\infty) \leq \|X_n-X_\infty\|_p$ yields the result.
\end{proof}

The proof of Theorem \ref{thm:liftmart} being quite involved, we present it separately in a section below.

We end this section by insisting upon the fact that, when working with $\mathbb L^p$ valued random variables, there are two notions of law which we need to distinguish with care. Let $X$ be a $\mathbb L^p$ valued random variable on a probability space $(\Omega, \mcl A , \mathbb P)$, and recall that $\mathbb L^p$ is the Lebesgue space of $\R^d$ valued Borel maps on the canonical space $([0,1], Leb)$. The first notion of a law of $X$ is what we can call its inner law. For every $\omega \in \Omega$, $X(\omega)$ is an element of $\mathbb L^p$ whose (inner) law is $X(\omega)_\#Leb$. The second notion is the outer law, it is the law of $X$ as a random variable on $\Omega$. It is the element $X_\#\mathbb P \in \mcl P(\mathbb L^p)$. Consequently, there are two sources of randomness in $X$, one coming from the outer space $\Omega$, and one from the inner space $[0,1]$.

\subsection{Deterministic non-constant martingales}
In this section, we discuss horizontal martingales which are non-constant but deterministic. We start with an example.\\

 Let $d=1$, $\eta_0= \delta _0$ and $\eta_1 = \frac12(\delta_{-1} + \delta_1)$. One has $(-Id)_\#\eta_1 = \eta_1$, which implies $(\frac12\pi_1 + \frac12 \pi_2)_\#(Id,-Id)_\#\eta_1 = \eta_0$. Thus, for any coupling convex function $\phi$, 
 $$
 \phi(\eta_0) \leq \phi(\eta_1).
 $$
 Thus $(\eta_0,\eta_1)$ is a horizontal martingale which is deterministic but not constant. Of course, such phenomenon cannot happen for martingales on Banach spaces, and this is mainly due to the fact that two distinct points can be separated by an hyperplane, which is defined by an affine function which is convex and concave. In this horizontal geometry, we do not have as many separating functions and it turns out that we cannot separate $\eta_0$ from $\eta_1$. This is obvious if one remarks that even though $\eta_0\ne \eta_1$, $\eta_0 \in \ol{cco}(\eta_1)$.\\
 
 The previous example generalizes instantly to any peacock process. Recall that $(µ_t)_{t \in T}$ is a peacock process if for all $s \leq t$, $µ_s \leq_c µ_t$. Such processes were extensively studied, for instance for their applications in mathematical finance, and we refer to Hirsh et al \citep{yorandco} for more details on them. In fact, we have the following result.
 \begin{Prop}
 The deterministic $p$-horizontal martingales are the $\mpprd$-valued peacock processes. 
 \end{Prop}
 \begin{proof}
 Let $(µ_t)_{t \in T}$ be a $p$-horizontal martingale. and let $\phi: \R^d \mapsto \R$ be a convex function. Then $µ \mapsto \int_{\R^d}\phi dµ$ is coupling convex. Approximating $\phi$ by convex functions growing at most like $|x|^p$ if necessary, we deduce that if $s\leq t$, then $\int_{\R^d}\phi dµ_s\leq \int_{\R^d}\phi dµ_t$. On the other hand, the result follows from Proposition \ref{prop:preliminaries}.
\end{proof}
 
  Note that the question of "lifting" peacock processes to martingales has been a well studied topic already, notably by Strassen \citep{strassen}, Kellerer \citep{kellerer} and Hirsch et al \citep{yorandco}. Somehow, in the next section, we study (for discrete times), an extension of such studies to horizontal martingales.\\

We end this section by insisting that such deterministic non-constant martingales yield standard stochastic martingales when lifted at the level of $\mpprd$ (this is obvious in our simple example). Our interpretation of this fact is that martingales taking values in $\mathbb L^p$ can have an inner form of stochasticity (inner because actually acting on the probability space onto which $\mathbb L^p$ is defined). When pulled down at the level of the probability measures, this stochasticity is transformed into a deterministic operation. While lifting horizontal martingales in the next section, we will need to extend the probability space to allow for such randomness to happen.

\subsection{Existence of lifted martingales}

The aim of this section is to establish Theorem \ref{thm:liftmart}. We start by proving the following elementary step.
\begin{Prop}\label{prop:element}
Let $p\in (1,\infty)$, $µ_0 \in \mpprd$ and consider a $\mpprd$ valued random variable $µ_1$ on a probability space $(\Omega, \mcl A,\mathbb P)$ such that $\mathbb E[(M_p(µ_1))^\frac1p] < \infty$, and such that for any continuous and coupling convex $\phi: \mpprd \mapsto \R$, 
\be\label{eq:martm0}
\phi(µ_0) \leq \mathbb E_{\mathbb P}[\phi(µ_1)].
\ee
Then, for any $X_0 \in \mathbb L^p$ of (inner) law $µ_0$, there exists an $\mathbb L^p$ valued random variable $X_1$ such that, for $\mathbb P$ almost every $\omega$, $X_1(\omega)_\#Leb = µ_1(\omega)$, together with 
$$
\mathbb E_{\mathbb P}[X_1] = X_0 \text{ in } \mathbb L^p.
$$
\end{Prop}

In order to construct such a couple, we pass through an intermediate step. Furthermore, in the statement we use the notation $\mathbb E_{\mathbb P}$ to insist upon the fact that the expectation is taken with respect to the outer probability space. 
\begin{Prop}\label{prop:step1}
Under the assumptions of Proposition \ref{prop:element}, for any $X_0 \in \mathbb L^p$ such that $(X_0)_\#Leb = µ_0$, there exists a Borel map $b: \mpprd \mapsto \mathbb L^p$ such that, almost surely in $\omega$
\be\label{cvxdom}
b(µ_1(\omega))_\#Leb \leq_{c} µ_1(\omega),
\ee
and 
$$
X_0 = \mathbb E_{\mathbb P}[b(µ_1)].
$$
\end{Prop}
The proof is obtained, as usual in this type of questions, see \citep{cartier,strassen}, by a separation argument.
\begin{proof}
We denote by $\mcl M\subset \mathbb L^p$ the set $\{\mathbb E_{\mathbb P}[b(µ_1)] \,|\,b: \mpprd \mapsto \mathbb L^p, \text{ Borel }, \eqref{cvxdom} \text{ holds a.s.}\}$. We now prove that $\mcl M$ is convex and closed in $\mathbb L^p$. The convexity follows from the convexity of the constraint $\leq_c$. To prove that it is closed, let $(x_n)_{n\geq 0}$ be an $\mcl M$-valued sequence converging in $\mathbb L^p$ toward some $x^*$ and $b_n$ a Borel map associated to $x_n$. Denote by $\eta$ the law of $µ_1$. Because of \eqref{cvxdom}, it follows that $\eta$ almost surely, for all $n\geq 0$,
$$
\|b_n\|_{\mathbb L^p} \leq (M_p(µ))^\frac1p.
$$
Because the right-hand side is in $L^1$, we deduce that the sequence $(b_n)_{n \geq 0}$ (belonging to $L^1(\eta,\mathbb L^p)$) is uniformly integrable and bounded. Furthermore, it also implies that for each Borelian $A$ of $\mpprd$, 
$$
\left\|\int_A b_nd\eta\right\|_{\mathbb L^p} \leq \int_A (M_p(µ))^\frac1p \eta(dµ) < \infty.
$$
Thus we can apply Theorem IV.2.1 in \citep{diestel} to obtain that $(b_n)_{n\geq 0}$ is relatively weakly compact in $L^1(\eta,\mathbb L^p)$ (because $\mathbb L^p$ has the RNP). The constraint \eqref{cvxdom} is only stable under strong convergence, but Mazur's Lemma implies that, because of its convexity, it is also stable under weak convergence. Since $b_n \mapsto x_n$ is weakly continuous, we thus obtain that $x^* \in \mcl M$ and thus that the latter is closed.

 Furthermore, $\mcl M$ is non-empty as we can choose $b(µ)$ to be the first moment of $µ$, which satisfies \eqref{cvxdom}. Assume $X_0\notin \mcl M$, then, using Hahn-Banach separation Theorem, there exists $Y \in \mathbb L^{p'}$ (deterministic) such that 
\be\label{eq:HB}
\langle Y , X_0 \rangle > \sup_{X'\in \mcl M} \langle Y, X'\rangle.
\ee
We now claim that 
\be\label{eq:696}
\sup_{X'\in \mcl M} \langle Y, X'\rangle = \mathbb E_{\mathbb P}[-\mcl I_p(µ_1,\mcl L(Y))].
\ee
In order to obtain this equality, fix $\eps > 0$ and consider a Borel map $b_\eps:  \mpprd \mapsto \mathbb L^p$ such that for all $µ$, $b_\eps(µ) \in \{ X \in \mathbb L^p, \mathcal L(X) = µ, \langle X,Y\rangle \geq -\mcl I_p(µ,\mcl L(Y)) - \eps\}=: B_\eps(µ)$. Such a measurable map exists because of Kuratowski selection Theorem, since the set $B_\eps(µ)$ is non-empty, closed, the map $B_\eps$ is measurable and valued in a Polish space. From this we deduce that 
$$
\sup_{X'\in \mcl M} \langle Y\cdot X'\rangle \geq \mathbb E_{\mathbb P}[b(µ_1)\cdot Y] \geq \mathbb E_{\mathbb P}[-\mcl I_p(µ_1,\mcl L(Y))] - \eps.
$$
Since $\eps > 0$ is arbitrary, we obtain the $\geq$ inequality in \eqref{eq:696}. To obtain the reverse inequality, we start with the computation
$$
\sup_{X'\in \mcl M} \langle Y, X'\rangle = \sup_{b, \eqref{cvxdom} \text{ holds }} \mathbb E_{\mathbb P}[\langle Y, b(µ_1)\rangle] \leq \sup_{b, \eqref{cvxdom}\text{ holds }} \mathbb E_{\mathbb P}[-\mcl I_p(b(µ_1)_\#Leb,\mcl L(Y))].
$$
Since $\mcl I_p(\cdot,\nu)$ is coupling concave and almost surely $b(µ_1)_\#Leb\leq_c µ_1$, the reverse inequality is thus a consequence of Proposition \ref{prop:preliminaries}. Using \eqref{eq:696} in \eqref{eq:HB} and using \eqref{eq:martm0} with the coupling convex function $-\mcl I_p(\cdot, \mcl L(Y))$, we arrive at 
$$
\langle Y , X_0 \rangle > - \mcl I_p(µ_0,\mcl L(Y)),
$$
which is a contradiction, hence the result is proved.
\end{proof}
\begin{Rem}\label{rem:separablefamily}
Note that only the family of coupling convex functions $(-\mcl I_p(\cdot,\nu))_{\nu \in \mcl P_{p'}(\R^d)}$ is used in the previous proof, which echoes with Theorem \ref{thm:repconvex}.
\end{Rem}
\begin{Rem}
To establish directly Theorem \ref{thm:liftmart}, we would have liked to replace $\mcl M$ by a set with a similar definition, except for replacing $\leq_{c}$ by an equality in \eqref{cvxdom}. However, such a set would not be convex, and thus not proper for a separation argument.
\end{Rem}
\newpage
As explained in the previous remark, we had to relax the constraint required on $X_1$, and this lead us to the existence of an intermediate random variable $b(µ_1(\cdot))$. To prove existence of a suitable $X_1$, we now make use of Strassen's Theorem (see Section 5 in \citep{strassen}), which we now recall in the particular case in which we shall use it.
\begin{Theorem}[Strassen]\label{thm:strassen}
Let $p \in [1,\infty)$, $µ,\nu \in \mpprd$. Then, $µ \leq_c \nu$ if and only if there exists a martingale coupling $\gamma$ between $µ$ and $\nu$, i.e. $\gamma \in \Pi(µ,\nu)$ and for all couple $(X,Y)$ of law $\gamma$, $\mathbb E[Y|X]= X$.
\end{Theorem}
For more on this type of results, we also refer to \citep{cartier,blackwell}.
\begin{proof}[Proof of Proposition \ref{prop:element}]
Let $X_0\in \mathbb L^p$ be of law $µ_0$ and consider the map $b$ given by Proposition \ref{prop:step1}. As stated just above, we want to use Strassen's Theorem to go from $b(µ_1)$ to some $X_1$. Formally, conditioned on $µ_1$, because $b(µ_1)_\#Leb \leq_{c} µ_1$, we know that there is a martingale coupling $\gamma_{µ_1}$ between the two previous probability measures (Theorem \ref{thm:strassen}). We denote by $\xi_{µ_1,X}$ the disintegration of $\gamma_{µ_1}$ along its first variable. Assuming that there exists a uniform on $[0,1]$ random variable $V$ defined on the inner probability space, independent from $b(µ_1)$, we can then deduce the existence of the required $X_1$. Indeed, using Lemma 2.22 in \citep{kallenberg}, there exists a measurable map $G_{µ_1}: \mathbb R^d\times [0,1]\mapsto \R^d$ such that for every $X \in \mathbb L^p$, and $\tilde V\sim \mcl U([0,1])$,  $G_{µ_1}(X,\tilde V)$ has law $\xi_{µ_1,X}$. In particular, we enlarge the probability space $\Omega$ to $\Omega\times \tilde \Omega$, where $\tilde \Omega= [0,1]$ with Lebesgue measure, and we set $X_1(\omega,\tilde \omega)(s) = G_{µ_1(\omega)}(b(µ_1(\omega))(s),V(s) +. \tilde \omega)$ to obtain the result, where $+.$ is the addition modulo $1$.

 In order to make the previous argument rigorous, there are four things to verify. First, that the map $µ_1\mapsto \gamma_{µ_1}$ can be chosen measurable. Second, that the same can be done for the map $G: µ_1 \mapsto G_{µ_1}$, third that there exists a random variable $V$, defined on the inner probability space, uniform on $[0,1]$, independent from $b(µ_1(\omega))$ for any $\omega \in \Omega$, and fourth that $X_1$ indeed satisfies $\mathbb E_{\mathbb P} [X_1] = X_0$ since $\mcl L(X_1) = µ_1$ is satisfied by construction.

The first step follows from Kuratowski selection result. Because $b$ is measurable, so is $µ \mapsto b(µ)_\#Leb$. Then, we obtain that the set valued map $µ \mapsto \{\gamma\in \Pi_M(\theta_µ,µ)\}$ is also measurable, where $\Pi_M$ denotes the set of martingale couplings. Furthermore it has closed values. Using once again the separability of the Wasserstein spaces, we deduce that it has a measurable selection which we denote by $\gamma_µ$. This "measurable Strassen" step is established in full details as Theorem 1.3 in \citep{leskela}.

The second point follows from the way Lemma 2.22 is proved in \citep{kallenberg}: it only requires a Borel isomorphism given once and for all and the measurability of the kernel given by the disintegration theorem, which holds.

The third point is not always true, but using the usual isomorphism between $[0,1]$ and $[0,1]^2$, we can always assume that our elements in $\mathbb L^p$ are defined on the inner probability space $[0,1]^2$ with Lebesgue measure, and that the map $b$ given by the previous Proposition is such that $b(µ_1(\omega))$ depends only on the first coordinate of the inner variable, for every $\omega$. Thus we can consider $V$ defined as $\pi_2$ to have a $\mcl U([0,1])$ random variable, independent (on the inner probability space) from $b(µ_1(\omega))$ for any $\omega$. Observe that this is equivalent to extending the inner probability space.

The last point follows from the computation, 
$$
\ba
\mathbb E_{\mathbb P}[X_1(s)] &= \int_{\Omega\times \tilde\Omega}G_{µ_1(\omega)}(b(µ_1(\omega))(s),V(s) +. \tilde \omega)\mathbb P(d\omega)d\tilde  \omega\\
&= \int_{\Omega\times \tilde\Omega}G_{µ_1(\omega)}(b(µ_1(\omega))(s), \tilde \omega)\mathbb P(d\omega)d\tilde  \omega\\
&= \int_{\Omega}b(µ_1(\omega))(s)\mathbb P(d\omega)= \mathbb E_{\mathbb P}[b(µ_1)] = X_0.
\ea
$$

\end{proof}
\begin{Rem}
Answering the previous remark, we were able to complete the proof by extending the probability space and use what is a measurable version of Strassen's Theorem to realize the "extra" randomness to pass from $b(µ_1)$ to an element  whose law is actually $µ_1$ almost surely.
\end{Rem}

From Proposition \ref{prop:element} (and its proof), we can then establish the following.
\begin{Prop}
Let $B_0$ and $B_1$ be two sub $\sigma$-algebras of a probability space $(\Omega, \mcl A, \mathbb P)$ such that $B_0\subset B_1$. Let $p\in (1,\infty)$ and $µ_0$ (resp. $µ_1$) be $B_0$ adapted (resp. $B_1$ adapted) $\mpprd$ valued random variable such that $\mathbb E_{\mathbb P}[(M_p(µ_i))^\frac1p] < \infty$ for $i= 0,1$. Let $X_0$ be a $B_0$ adapted, $\mathbb L^p$-valued random variable such that almost surely $(X_0)_\#Leb = µ_0$. If for any continuous and coupling convex $\phi: \mpprd \mapsto \R$, almost surely,
\be\label{eq:martm02}
\phi(µ_0) \leq \mathbb E_{\mathbb P}[\phi(µ_1)|B_0].
\ee
Then, there exist an extended probability space $(\tilde \Omega, \tilde{ \mcl A}, \tilde {\mathbb P})$, a sub $\sigma$-algebra $A_1$ of $\tilde{\mcl A}$ such that $B_0\subset A_1$, and $X_1$, a $A_1$ adapted $\mathbb L^p$-valued random variables, such that almost surely, $µ_1 = (X_1)_\#Leb$ and, almost surely,
$$
\mathbb E[X_1|B_0] = X_0 \text{ in } \mathbb L^p.
$$
\end{Prop}
\begin{proof}
The proof is an adaptation of the one of Proposition \ref{prop:element}. We indicate the main changes but do not give all details. First, because of Remark \ref{rem:separablefamily}, \eqref{eq:martm02} only needs to hold for coupling convex functions of the form $(-\mcl I_p(\cdot,\xi))_{\xi \in \mcl P_{p'}(\R^d)}$. Because of the separability of $\mcl P_{p'}(\R^d)$ and the regularity of $\mcl I_p$ in its second argument, we only need to verify \eqref{eq:martm02} for a countable family. Hence, almost surely, \eqref{eq:martm02} holds for any coupling convex $\phi$.

To construct the analogue of the map $b$ of Proposition \ref{prop:step1}, we consider instead a measurable function $\beta$ of $\Omega$ itself, and obtain it by considering the space $\mcl M^*$ defined by $\mcl M^* := \{\mathbb E[\beta|B_0] \,|\, \beta: \Omega\mapsto \mathbb L^p, \text{ Borel }, \beta_\#Leb \leq_c µ_1 \text{ a.s.}\}\subset L^1(\Omega,\mathbb L^p)$, which is a separable Banach space. To apply a separation argument, we need to verify that $\mcl M^*$ is indeed closed, convex and non-empty. The arguments here followed the one we gave for $\mcl M$ above.

The passage from $\beta$ to $X_1$ then follows as in the previous case.

\end{proof}

From the previous proposition, Theorem \ref{thm:liftmart} simply follows by induction. We insist upon the fact that our construction of the lifted martingale $(X_n)_{n \geq 0}$ requires a possible extension of the probability space at each stage $n$. Notably, the martingale $(X_n)_{n \geq 0}$ is defined on a filtration which is quite larger than the one on which $(µ_n)_{n \geq 0}$ is defined.

\subsection*{Acknowledgments}
Funded by the European Union (ERC, PaDiESeM, 101222038). Views and opinions expressed are however those of the author only and do not necessarily reflect those of the European Union or the European Research Council. Neither the European Union nor the granting authority can be held responsible for them.
The author acknowledges a partial support from the Chair FDD (Institut Louis Bachelier).

\bibliographystyle{plainnat}
\bibliography{bibwas}

\begin{thebibliography}{42}
\providecommand{\natexlab}[1]{#1}
\providecommand{\url}[1]{\texttt{#1}}
\expandafter\ifx\csname urlstyle\endcsname\relax
  \providecommand{\doi}[1]{doi: #1}\else
  \providecommand{\doi}{doi: \begingroup \urlstyle{rm}\Url}\fi

\bibitem[Agueh and Carlier(2011)]{carlier}
Martial Agueh and Guillaume Carlier.
\newblock Barycenters in the wasserstein space.
\newblock \emph{SIAM Journal on Mathematical Analysis}, 43\penalty0
  (2):\penalty0 904--924, 2011.

\bibitem[Alfonsi and Jourdain(2020)]{alfonsi}
Aur{\'e}lien Alfonsi and Benjamin Jourdain.
\newblock Squared quadratic {W}asserstein distance: optimal couplings and
  {L}ions differentiability.
\newblock \emph{ESAIM: Probability and Statistics}, 24:\penalty0 703--717,
  2020.

\bibitem[Aliaga(2026)]{aliaga2026}
Ram{\'o}n~J Aliaga.
\newblock Lipschitz-free spaces and purely 1-unrectifiable metric spaces.
\newblock \emph{arXiv preprint arXiv:2606.02918}, 2026.

\bibitem[Bayraktar et~al.(2023)Bayraktar, Ekren, and Zhang]{bayraktar}
Erhan Bayraktar, Ibrahim Ekren, and Xin Zhang.
\newblock A smooth variational principle on wasserstein space.
\newblock \emph{Proceedings of the American Mathematical Society}, 151\penalty0
  (09):\penalty0 4089--4098, 2023.

\bibitem[Bellini et~al.(2021)Bellini, Koch-Medina, Munari, and
  Svindland]{bellini2021law}
Fabio Bellini, Pablo Koch-Medina, Cosimo Munari, and Gregor Svindland.
\newblock Law-invariant functionals on general spaces of random variables.
\newblock \emph{SIAM Journal on Financial Mathematics}, 12\penalty0
  (1):\penalty0 318--341, 2021.

\bibitem[Bertucci(2023)]{bertucci2023monotone}
Charles Bertucci.
\newblock Monotone solutions for mean field games master equations: continuous
  state space and common noise.
\newblock \emph{Communications in Partial Differential Equations}, 48\penalty0
  (10-12):\penalty0 1245--1285, 2023.

\bibitem[Bertucci(2025)]{bertucci2025tangent}
Charles Bertucci.
\newblock The tangent space to the {Wasserstein} space: parallel transport and
  other applications.
\newblock \emph{arXiv preprint arXiv:2512.09763}, 2025.

\bibitem[Bertucci(2026)]{bertuccibook}
Charles Bertucci.
\newblock Analysis on spaces of measures.
\newblock \emph{arXiv preprint arXiv:2607.19939}, 2026.

\bibitem[Bertucci and Lions(2026)]{nonstegall}
Charles Bertucci and Pierre-Louis Lions.
\newblock Non-linear {S}tegall's lemma and general {Hamilton-Jacobi-Bellman}
  equations on {Wasserstein} spaces.
\newblock \emph{arXiv preprint arXiv:2606.31297}, 2026.

\bibitem[Blackwell(1953)]{blackwell}
David Blackwell.
\newblock Equivalent comparisons of experiments.
\newblock \emph{The annals of mathematical statistics}, pages 265--272, 1953.

\bibitem[Borwein and Preiss(1987)]{BP}
Jonathan~M Borwein and David Preiss.
\newblock A smooth variational principle with applications to
  subdifferentiability and to differentiability of convex functions.
\newblock \emph{Transactions of the American Mathematical Society},
  303\penalty0 (2):\penalty0 517--527, 1987.

\bibitem[Bourgain(1977)]{bourgain2}
Jean Bourgain.
\newblock On dentability and the {Bishop-Phelps} property.
\newblock \emph{Israel Journal of Mathematics}, 28\penalty0 (4):\penalty0
  265--271, 1977.

\bibitem[Bourgain(1980)]{bourgain1}
Jean Bourgain.
\newblock Dunford-pettis operators on l 1 and the radon-nikodym property.
\newblock \emph{Israel Journal of Mathematics}, 37\penalty0 (1):\penalty0
  34--47, 1980.

\bibitem[Cardaliaguet et~al.(2019)Cardaliaguet, Delarue, Lasry, and
  Lions]{cdll}
Pierre Cardaliaguet, Fran{\c{c}}ois Delarue, Jean-Michel Lasry, and
  Pierre-Louis Lions.
\newblock \emph{The Master Equation and the Convergence Problem in Mean Field
  Games:(AMS-201)}, volume 201.
\newblock Princeton University Press, 2019.

\bibitem[Carmona et~al.(2018)Carmona, Delarue,
  et~al.]{carmona2018probabilistic}
Ren{\'e} Carmona, Fran{\c{c}}ois Delarue, et~al.
\newblock \emph{Probabilistic Theory of Mean Field Games with Applications
  I-II}.
\newblock Springer, 2018.

\bibitem[Cartier et~al.(1964)Cartier, Fell, and Meyer]{cartier}
Pierre Cartier, JMG Fell, and Paul-Andr{\'e} Meyer.
\newblock Comparaison, des mesures port{\'e}es par un ensemble convexe compact.
\newblock \emph{Bulletin de la Soci{\'e}t{\'e} Math{\'e}matique de France},
  92:\penalty0 435--445, 1964.

\bibitem[Cavagnari et~al.(2026)Cavagnari, Savar{\'e}, and Sodini]{cavagnari}
Giulia Cavagnari, Giuseppe Savar{\'e}, and Giacomo~Enrico Sodini.
\newblock A {Lagrangian approach to totally dissipative evolutions in
  {W}asserstein spaces}.
\newblock \emph{Journal of Differential Equations}, 470:\penalty0 114395, 2026.

\bibitem[Cosso et~al.(2024)Cosso, Gozzi, Kharroubi, Pham, and
  Rosestolato]{cosso2021master}
Andrea Cosso, Fausto Gozzi, Idris Kharroubi, Huy{\^e}n Pham, and Mauro
  Rosestolato.
\newblock Master bellman equation in the wasserstein space: Uniqueness of
  viscosity solutions.
\newblock \emph{Transactions of the American Mathematical Society},
  377\penalty0 (01):\penalty0 31--83, 2024.

\bibitem[Darling(1982)]{darling}
Richard~WR Darling.
\newblock Martingales in manifolds. definition, examples and behaviour under
  maps.
\newblock \emph{S{\'e}minaire de probabilit{\'e}s de Strasbourg}, 16:\penalty0
  217--236, 1982.

\bibitem[Diestel and Uhl(1977)]{diestel}
Joseph Diestel and John~J Uhl.
\newblock \emph{Vector measures}.
\newblock Mathematical surveys and monograph, vol 15. American Mathematical
  Society, 1977.

\bibitem[Ekeland(1974)]{ekeland}
Ivar Ekeland.
\newblock On the variational principle.
\newblock \emph{Journal of Mathematical Analysis and Applications}, 47\penalty0
  (2):\penalty0 324--353, 1974.

\bibitem[Gangbo and Tudorascu(2019)]{gangbotudorascu}
Wilfrid Gangbo and Adrian Tudorascu.
\newblock On differentiability in the {W}asserstein space and well-posedness
  for {Hamilton--J}acobi equations.
\newblock \emph{Journal de Math{\'e}matiques Pures et Appliqu{\'e}es},
  125:\penalty0 119--174, 2019.

\bibitem[Godefroy(2015)]{lfs}
Gilles Godefroy.
\newblock A survey on {Lipschitz-free Banach spaces}.
\newblock \emph{Commentationes Mathematicae}, 55\penalty0 (2), 2015.

\bibitem[Hirsch et~al.(2011)Hirsch, Profeta, Roynette, and Yor]{yorandco}
Francis Hirsch, Christophe Profeta, Bernard Roynette, and Marc Yor.
\newblock \emph{Peacocks and associated martingales, with explicit
  constructions}.
\newblock Springer Science \& Business Media, 2011.

\bibitem[Kallenberg(1997)]{kallenberg}
Olav Kallenberg.
\newblock \emph{Foundations of modern probability}.
\newblock Springer, 1997.

\bibitem[Kellerer(1972)]{kellerer}
Hans~G Kellerer.
\newblock Markov-komposition und eine anwendung auf martingale.
\newblock \emph{Mathematische Annalen}, 198\penalty0 (3):\penalty0 99--122,
  1972.

\bibitem[Lacker(2023)]{lacker}
Daniel Lacker.
\newblock Hierarchies, entropy, and quantitative propagation of chaos for mean
  field diffusions.
\newblock \emph{Probability and mathematical physics}, 4\penalty0 (2):\penalty0
  377--432, 2023.

\bibitem[Leskel{\"a} and Vihola(2017)]{leskela}
Lasse Leskel{\"a} and Matti Vihola.
\newblock {Conditional convex orders and measurable martingale couplings}.
\newblock \emph{Bernoulli}, 23\penalty0 (4A):\penalty0 2784 -- 2807, 2017.

\bibitem[Lions(2011-2012)]{lions20112012}
Pierre-Louis Lions.
\newblock Cours au college de france.
\newblock \emph{www.college-de-france.fr}, 2011-2012.

\bibitem[Meyer(2006)]{sanslarmes}
Paul-Andr{\'e} Meyer.
\newblock G{\'e}om{\'e}trie stochastique sans larmes.
\newblock In \emph{S{\'e}minaire de Probabilit{\'e}s XV 1979/80: Avec table
  g{\'e}n{\'e}rale des expos{\'e}s de 1966/67 {\`a} 1978/79}, pages 44--102.
  Springer, 2006.

\bibitem[Phelps(1988)]{phelps}
Robert~R Phelps.
\newblock \emph{Convex functions, monotone operators and differentiability},
  volume 1364.
\newblock Springer, 1988.

\bibitem[Pinzi and Savar{\'e}(2025)]{pinzi}
Alessandro Pinzi and Giuseppe Savar{\'e}.
\newblock Totally convex functions, {$ L^2$-Optimal transport for laws of
  random measures, and solution to the M}onge problem.
\newblock \emph{arXiv preprint arXiv:2509.01768}, 2025.

\bibitem[Pisier(2016)]{pisier}
Gilles Pisier.
\newblock \emph{Martingales in {B}anach spaces}.
\newblock Cambridge {U}niversity press, 2016.

\bibitem[Renesse and Sturm(2009)]{sturm2}
Max-K~von Renesse and Karl-Theodor Sturm.
\newblock Entropic measure and {W}asserstein diffusion.
\newblock \emph{The annals of probability}, 37\penalty0 (3), 2009.

\bibitem[Ryff(1965)]{ryff}
John~V Ryff.
\newblock {Orbits of $L^1$-functions under doubly stochastic transformations}.
\newblock \emph{Transactions of the American Mathematical Society},
  117:\penalty0 92--100, 1965.

\bibitem[Schiavo(2024)]{schiavo}
Lorenzo~Dello Schiavo.
\newblock Massive particle systems, wasserstein brownian motions, and the
  dean-kawasaki equation.
\newblock \emph{arXiv preprint arXiv:2411.14936}, 2024.

\bibitem[Schwartz(1980)]{schwartz}
Laurent Schwartz.
\newblock \emph{Semi-martingales sur des vari{\'e}t{\'e}s, et martingales
  conformes sur des vari{\'e}t{\'e}s analytiques complexes}.
\newblock Springer, 1980.

\bibitem[Stegall(1978)]{stegall}
Charles Stegall.
\newblock Optimization of functions on certain subsets of banach spaces.
\newblock \emph{Mathematische Annalen}, 236\penalty0 (2):\penalty0 171--176,
  1978.

\bibitem[Strassen(1965)]{strassen}
Volker Strassen.
\newblock The existence of probability measures with given marginals.
\newblock \emph{The Annals of Mathematical Statistics}, 36\penalty0
  (2):\penalty0 423--439, 1965.

\bibitem[Sturm(2002)]{sturm}
Karl-Theodor Sturm.
\newblock Nonlinear martingale theory for processes with values in metric
  spaces of nonpositive curvature.
\newblock \emph{The Annals of Probability}, 30\penalty0 (3):\penalty0
  1195--1222, 2002.

\bibitem[Sturm(2026)]{sturm3}
Karl-Theodor Sturm.
\newblock {Wasserstein diffusion on multidimensional spaces}.
\newblock \emph{The Annals of Probability}, 54\penalty0 (2):\penalty0 610 --
  643, 2026.

\bibitem[Villani(2009)]{villani}
C{\'e}dric Villani.
\newblock \emph{Optimal transport: old and new}, volume 338.
\newblock Springer, 2009.

\end{thebibliography}

\appendix \section{Standard proofs}
\textbf{Proof of Theorem \ref{thm:RNP}:}

The proof follows Theorem V.3.7 in \citep{diestel}. We start by applying the Radon-Nikodym Theorem (in $\R$), to the finite measure $\|G\|$ and we denote by $\phi$ its derivative with respect to $µ$. We set for $n \geq 1$, $A_n = \{\omega | n-1 \leq \phi (\omega) < n\}$. Clearly, $(A_n)_{n \geq 1}$ forms a partition of $\Omega$ and, for every $n \geq 1$, for all $A \subset A_n$
$$
\frac{|G|(A)}{µ(A)} \leq n.
$$
Let $\eps > 0$ and $A \in \mcl A$ such that $µ(A) > 0$. There exists $n \geq 1$ such that $µ(A_n\cap A) > 0$. Thus, there exists $A' \in \mcl A$ such that $\mcl O := \{\frac{G(B)}{µ(B)}, B \subset A', µ(B) > 0\}$ is bounded in $E$, where in the previous set and in all what follows, we are only concerned with sets which belong to $\mcl A$. Furthermore, it is a subset of $\mcl K$. Thus, by using Proposition \ref{prop:dent} (on $\mcl K$), we know that there exists $C \subset A'$, $µ(C) > 0$ such that 
\be\label{eq:star}
\frac{G(C)}{µ(C)} \notin \ol{co}(\mcl O\setminus B(\frac{G(C)}{µ(C)},\eps)).
\ee 
If 
$$
\sup\left\{\left\|\frac{G(E)}{µ(E)} - \frac{G(C)}{µ(C)} \right\|, E \subset C\right\} \leq \eps,
$$
we stop and use Lemma \ref{lemma:X} below. If not, let $j_1$ be the smallest integer $\geq 2$ such that 
$$
\exists C_1 \subset C, µ(C_1) \geq \frac{1}{j_1} \text{ and } \left\|\frac{G(C_1)}{µ(C_1)} - \frac{G(C)}{µ(C)} \right\| > \eps.
$$
Because $G$ is a vector measure,
$$
\frac{G(C)}{µ(C)} = \frac{G(C_1)}{µ(C_1)}\frac{µ(C_1)}{µ(C)} + \frac{G(C\setminus C_1)}{µ(C\setminus C_1)}\frac{µ(C\setminus C_1)}{µ(C)}.
$$
Now, if $\sup\left\{\left\|\frac{G(E)}{µ(E)} - \frac{G(C)}{µ(C)} \right\|, E \subset C\setminus C_1\right\} \leq \eps$, we stop and use Lemma \ref{lemma:X} below, else we continue in the same way, and if this process stops, we use Lemma \ref{lemma:X} below, otherwise, we obtain sequences $(C_n)_{n \geq 1}$ and $(j_n)_{n \geq 1}$ such that
\begin{itemize}
\item $\left\|\frac{G(C_n)}{µ(C_n)} - \frac{G(C)}{µ(C)} \right\| \geq \eps$ for all $n \geq 1$,
\item if $E \subset C \setminus (\cup_{i=1}^m C_i)$, $\left\|\frac{G(E)}{µ(E)} - \frac{G(C)}{µ(C)} \right\| > \eps$, then $µ(E) < \frac{1}{j_m-1},$
\item $$
\frac{G(C)}{µ(C)} =  \frac{G(C\setminus (\cup_{n=1}^mC_n))}{µ(C\setminus (\cup_{n=1}^mC_n))}\frac{µ(C\setminus (\cup_{n=1}^mC_n))}{µ(C)} + \sum_{n = 1}^m\frac{G(C_n)}{µ(C_n)}\frac{µ(C_n)}{µ(C)}
$$
\end{itemize}
Recall that, as $ m \to \infty$, $\frac{G(C \setminus (\cup_{n=1}^mC_n))}{µ(C\setminus (\cup_{n=1}^mC_n))}$ is bounded, thus $\lim_{m \to \infty} µ(C\setminus (\cup_{n=1}^mC_n)) > 0$, otherwise we contradict \eqref{eq:star}. Write $B = C \setminus (\cup_{n=1}^\infty C_n)$. We now argue that 
$$
\sup\left\{\left\|\frac{G(E)}{µ(E)} - \frac{G(C)}{µ(C)} \right\|, E \subset B, µ(E) > 0\right\} \leq \eps.
$$
Indeed, if it is not the case, $\exists E \subset C\setminus (\cup_{i=1}^mC_i)$ for all $m \geq 1$, and $\left\|\frac{G(E)}{µ(E)} - \frac{G(C)}{µ(C)} \right\|> \eps$, so that by the second point above, $µ(E) \leq \frac{1}{j_m-1}$ for all $m$. However, $ j_m\to \infty$ as $m \to \infty$. Indeed, $µ(C_m) \geq \frac{1}{j_m}$ and $(C_m)$ is a sequence of disjoint elements implies that $\sum_{m \geq 1} \frac{1}{j_m} \leq \sum_{m \geq 1}µ(C_m) < \infty$. In particular, $µ(E) =0$ which is a contradiction. Hence the proof is complete thanks to the next lemma.

\begin{Lemma}\label{lemma:X}
Let $G$ be a $µ$-continuous vector measure of bounded variation valued in $\mcl K$. Assume that for every $\eps > 0$ and $A \in \mcl A$, $µ(A) > 0$, there exists $B \subset A$, $B \in \mcl A$, $µ(B) > 0$ such that $\{\frac{G(E)}{µ(E)}, E \in \mcl A, E \subset B, µ(E) > 0\}$ has diameter at most $\eps$. Then, there exists $f \in L^1(µ,\mcl K)$ such that for all $E \in \mcl A$,
$$
G(E) =\int_{\Omega}f(\omega)µ(d\omega).
$$
\end{Lemma}
\begin{proof}
This result and proof follows Lemma V.3.6 in \citep{diestel}. Let $\eps > 0$. By the exhaustion Lemma, there exists a sequence $(E_n(\eps))_{n \geq 1}$ valued in $\mcl A$ with $µ(E_n(\eps)) > 0$ for all $n$, $µ(\Omega \setminus (\cup_{i = 1}^\infty E_i(\eps))) = 0$, and $\{\frac{G(E)}{µ(E)}, E \subset E_n(\eps), µ(E) > 0, E \in \mcl A\}$ has diameter at most $\eps$. Define $f_\eps: \Omega \mapsto \mcl K$ by $f_\eps = \sum_{n =1}^\infty \frac{G(E_n(\eps))}{µ(E_n(\eps))} \chi_{E_n(\eps)}$ and $F_\eps$ the vector measure defined for $A \in \mcl A$ by
$$
F_\eps(A) = \int_A f_\eps(\omega)µ(d\omega).
$$
For $\pi$ a measurable countable partition of $\Omega$,
$$
\sum_{E \in \pi} \|F(E) - F_\eps(E)\| \leq \sum_{E \in \pi} \sum_{n=1}^\infty \left\|\frac{F(E\cap E_n(\eps))}{µ(E\cap E_n(\eps))} - \frac{F(E_n(\eps))}{µ(E_n(\eps))}  \right\|µ(E\cap E_n(\eps)) \leq \eps µ(\Omega).
$$
From this estimate, we easily obtain that $(f_\eps)$ is a Cauchy sequence. Its limit $f$ is of course valued in $\mcl K$ and satisfies the claim.
\end{proof}

\end{document}